\documentclass[ECP]{ejpecp_bdz}

\AUTHORS{%
  Dmitry~Ostrovsky
}

\SHORTTITLE{On Barnes Beta Distributions}
\TITLE{On Barnes Beta Distributions, Selberg Integral and Riemann Xi
}

\AMSSUBJ{Primary: 33B15; 60E07; 60E10; Secondary: 11M06; 40A20; 60G57}

\VOLUME{0} \YEAR{2014} \PAPERNUM{0} \DOI{vVOL-PID}  

\ABSTRACT{The theory of Barnes beta probability distributions is advanced and related to the Riemann xi function.
The scaling invariance, multiplication formula, and Shintani factorization of Barnes multiple gamma functions are reviewed
using the approach of Ruijsenaars and shown to imply novel properties of Barnes beta distributions. The applications are given to the meromorphic extension of the Selberg integral as a function of its dimension and the scaling invariance of the underlying probability distribution. This probability distribution in the critical case is described and conjectured to be the distribution of the derivative martingale.
The Jacobi triple product is interpreted probabilistically resulting in an approximation of the Riemann xi function by the Mellin transform of the logarithm of a limit of Barnes beta distributions.
}

\dedicated{Dedicated to the memory of my teachers Serge Lang, Benoit B. Mandelbrot, and Steven A. Orszag}

\begin{document}

\section{Introduction}

This paper continues our study of the family of Barnes beta probability\footnote{We use ``probability distribution'', ``law'', and ``random variable'' interchangeably in this paper.} distributions \(\beta_{M,N}\) on the unit interval \((0,1]\) that we initiated in \cite{Me} in the special case of \(M=N=2\) and developed in \cite{Me13} for general \(M\leq N\), \(M,N\in\mathbb{N}.\) Our primary motivation for introducing and studying Barnes beta distributions was their natural appearance in the meromorphic extension of the Selberg integral as a function of its dimension that we found in \cite{Me4} and \cite{Me} in our work on the law of the total mass of the limit lognormal stochastic measure on the unit interval. Aside from this particular application, they possess special properties that connect them with two areas of current interest in probability theory. In fact, the defining property of \(\beta_{M,N}\) is that its Mellin transform is given by a product of ratios of Barnes multiple gamma functions. Moreover, $-\log\beta_{M,N}$ is infinitely divisible on the non-negative real line. Hence, on the one hand, the study of this family of probability distributions is part of a general study of infinite divisibility in the context of special functions of analytic number theory complementing \cite{BiaPitYor}, \cite{Jacod}, \cite{LagRains}, \cite{NikYor}. On the other hand, \(\beta_{M,N}\) provides a natural generalization of Dufresne distributions,  confer \cite{ChaLet}, \cite{Duf10}, whose Mellin transform is of the form of a product of ratios of Euler's gamma functions, which corresponds to $M=1$ in our case. In addition to Barnes beta distributions being at the confluence of these two areas of probability, they have a rich mathematical structure that is of independent interest. Specifically, the Mellin transform of \(\beta_{M,N}\) satisfies several symmetries relating $\beta_{M,N}$ to $\beta_{M,N-1},$ $\beta_{M-1,N},$ and $\beta_{M-1,N-1}.$ It also satisfies a functional equation that leads to a remarkable infinite factorization of the Mellin transform (Shintani factorization) and to finite product formulas for the moments of \(\beta_{M,N}\) extending Selberg's formula.
The pole structure of the Mellin transform is complicated and depends on the rationality of the coefficients of \(\beta_{M,N}.\) In addition, the Laplace transform of \(\beta_{M,N}\) is given by an infinite series that gives an interesting generalization of the confluent hypergeometric function.

The contribution of this paper is threefold. First, we advance the general theory of Barnes beta distributions. We show that they possess a scaling invariance and give a new derivation of the Shintani factorization of the Mellin transform. We also show that this factorization implies a second infinite factorization of the Mellin transform (Barnes factorization), interpret it probabilistically,  
and give novel explicit formulas for the integral moments and mass at 1 of \(\beta_{M,N}\) as an application. 
In the special case of $M=N=2$ we prove that a certain product of ratios of double gamma functions is the Mellin transform of a probability
distribution, which factorizes into a product of \(\beta_{2,2}^{-1}s\) and a lognormal. This result gives a new and simpler approach to our key result in \cite{Me} on the extension of the Selberg integral as a function of its dimension to the Mellin transform of a probability distribution (conjectured to be the law of the total mass of the limit lognormal stochastic measure on the unit interval). 
We also show that our scaling invariance implies the invariance of the Mellin transform that was first noted in \cite{FLDR} and give it a novel probabilistic formulation.

Second, we study in three ways the relationship between the Barnes $G$ function,
Selberg integral, and Riemann xi function. This subject is of substantial interest in light of the work of \cite{KS} and, more recently, \cite{YK} that related these objects conjecturally. Our contribution is to show that Barnes beta distributions provide natural links that connect these objects. Specifically, first, by taking
a weak limit of the probability distribution underlying the Selberg integral
(confer above), we show that the Mellin transform of the resulting distribution
(conjectured to be the law of the derivative martingale, confer \cite{barraletal}, \cite{dupluntieratal}) is given by a product of ratios of Barnes $G$ functions. Second, we factorize this distribution into a product of a lognormal, \(\beta^{-1}_{2, 2}\), Pareto, and Fr\'echet distributions and observe that the L\'evy density of \(-\log\beta_{2, 2}\) is given by the Laplace transform of the $C_2$ distribution, confer \cite{BiaPitYor}, which is known to be directly linked to the xi function. In particular, this observation allows us to compute the cumulants of a class of \(-\log\beta_{2, 2}\) distributions in terms of the values of the xi function at the integers. Third,
motivated by \cite{BiaPitYor}, we introduce a new family of infinitely divisible distributions (extending the $S_2$ distribution, confer \cite{BiaPitYor}) and express it as the logarithm of a limit of Barnes beta distributions of the form
\(\beta_{2M,3M},\) \(M\rightarrow\infty,\) using Jacobi's triple product. We relate the Mellin transform of this family to the xi function by a functional equation thereby deriving a new approximation for xi.

Third, we contribute to the theory of Barnes multiple gamma functions. In the framework of the remarkable approach to multiple gamma functions due to Ruijsenaars \cite{Ruij}, we give new proofs of the scaling invariance, multiplication relation, and Shintani identity for these functions that play a key role in the proofs of our results on Barnes beta distributions. While these results are known in the classical case of the original Barnes normalization of gamma functions, confer \cite{mBarnes}, \cite{KataOhts}, \cite{Kuz}, and \cite{Shintani}, our formulation of them using the modern normalization is new. The classical normalization also unnecessarily complicated the proofs, while the approach of Ruijsenaars, which is based on his Malmst\'en-type formula for the log-gamma function, allows us to give elementary proofs that reduce to simple properties of multiple Bernoulli polynomials. In particular, our statement of the Shintani identity is simpler and more general  
than what is available in the existing literature. 

The plan of this paper is as follows. In Section 2 we give a brief review of Barnes multiple gamma functions based on \cite{Ruij} followed by our formulation of their scaling invariance, multiplication relation, and Shintani factorization. In Section 3 we give a review of the key properties of Barnes beta distributions that we established in \cite{Me13}. In Section 4 we state and prove our new results on general Barnes beta distributions and in Section 5 on the Selberg integral distribution. In Section 6 we study the critical case. In Section 7 we derive our approximation of the Riemann xi function. In Section 8 we give the proofs of the three properties of Barnes multiple gamma functions that are stated in Section 2. Section 9 concludes.

\section{A Review of Multiple Gamma Functions}

Let $f(t)$ be of the Ruijsenaars class, \emph{i.e.} analytic for
$\Re(t)>0$ and at $t=0$ and of at worst polynomial growth as
$t\rightarrow \infty,$ confer \cite{Ruij}, Section 2. The main
example that corresponds to the case of Barnes multiple gamma
functions is
\begin{equation}\label{fdef}
f(t) = t^M \prod\limits_{j=1}^M (1-e^{-a_j t})^{-1}
\end{equation}
for some integer $M\geq 0$ and parameters $a_j>0,$ $j=1\cdots M.$
For concreteness, the reader can assume with little loss of
generality that $f(t)$ is defined by \eqref{fdef}. Slightly
modifying the definition in \cite{Ruij}, we define multiple
Bernoulli polynomials by
\begin{equation}\label{Bdefa}
B^{(f)}_{m}(x) \triangleq \frac{d^m}{dt^m}|_{t=0} \bigl[f(t)
e^{-xt}\bigr]
\end{equation}
and in the case of \eqref{fdef} denote them by \(B_{M,m}(x\,|\,a).\)
The generalized zeta function is defined by
\begin{equation}\label{zdef}
\zeta^{(f)}_M(s, \,w) \triangleq \frac{1}{\Gamma(s)} \int\limits_0^\infty
t^{s-1} e^{-wt}\,f(t) \,\frac{dt}{t^M}, \,\,\Re(s)>M,\,\Re(w)>0.
\end{equation}
It is shown in \cite{Ruij} that $\zeta^{(f)}_M(s, \,w)$ has an analytic
continuation to a function that is meromorphic in $s\in\mathbb{C}$
with simple poles at $s=1, 2, \cdots M.$ The generalized log-gamma
function is then defined by
\begin{equation}\label{Ldef}
L^{(f)}_M(w) \triangleq \partial_s \zeta^{(f)}_M(s, \,w)|_{s=0}, \,\,\Re(w)>0.
\end{equation}
It can be analytically continued to a function that is holomorphic
over $\mathbb{C}-(-\infty, 0].$ The key result of \cite{Ruij} that
we need is summarized in the following theorem.
\begin{theorem}[Ruijsenaars]\label{R}
$L^{(f)}_M(w)$ satisfies the Malmst\'en-type formula for $\Re(w)>0,$
\begin{equation}\label{key}
L^{(f)}_M(w) = \int\limits_0^\infty \frac{dt}{t^{M+1}} \Bigl(
e^{-wt}\,f(t) - \sum\limits_{k=0}^{M-1} \frac{t^k}{k!}\,B^{(f)}_k(w)
- \frac{t^M\,e^{-t}}{M!}\, B^{(f)}_M(w)\Bigr).
\end{equation}
\end{theorem}

In the special case of the function $f(t)$ being defined by
\eqref{fdef},\footnote{Whenever $f(t)$ is defined by
\eqref{fdef} or \eqref{fgdef} below, the index $M$ in \eqref{fdef} is always the same as the subscript $M$ in \eqref{zdef} and \eqref{Ldef}.} the generalized zeta and log-gamma functions, which in this case we denote by \(\zeta_M\bigl(s,\,w\,|\,a\bigr)\) and \(L_M(w\,|\,a)\) respectively, have important additional properties. It is not difficult to show that
\eqref{zdef} becomes
\begin{equation}\label{mzdef}
\zeta_M\bigl(s,\,w\,|\,a\bigr) =
\sum\limits_{k_1,\cdots,k_M=0}^\infty \bigl(w+k_1 a_1+\cdots+k_M
a_M\bigr)^{-s},\,\,\Re(s)>M,\,\Re(w)>0,
\end{equation}
for $a=(a_1,\cdots, a_M),$ which is the formula given originally by
Barnes \cite{mBarnes} for the multiple zeta function.
Now, following \cite{Ruij}, define the Barnes multiple gamma function by
\begin{equation}\label{mgamma}
\Gamma_M(w\,|\,a) \triangleq \exp\bigl(L_M(w\,|\,a)\bigr).
\end{equation}
It follows from \eqref{mzdef} and \eqref{mgamma} that
$\Gamma_M(w\,|\,a)$ satisfies the fundamental functional equation
\begin{equation}\label{feq}
\Gamma_{M}(w\,|\,a) =
\Gamma_{M-1}(w\,|\,\hat{a}_i)\,\Gamma_M\bigl(w+a_i\,|\,a\bigr),\,i=1\cdots
M, \,\,M=1, 2, 3\cdots,
\end{equation}
$\hat{a}_i = (a_1,\cdots, a_{i-1},\,a_{i+1},\cdots, a_{M}),$ and
$\Gamma_0(w) = 1/w,$ which is also due to \cite{mBarnes}.
By iterating \eqref{feq} one sees that $\Gamma_M(w\,|\,a)$ is
meromorphic over $\mathbb{C}$ having no zeroes and poles at
\begin{equation}\label{poles}
w=-(k_1 a_1+\cdots + k_M a_M),\; k_1\cdots k_M\in\mathbb{N},
\end{equation}
with multiplicity equal the number of $M-$tuples $(k_1, \cdots,
k_M)$ that satisfy \eqref{poles}.

We conclude our review of the Barnes functions by relating the
general results to the classical case of Euler's gamma and Hurwitz's
zeta functions. Following \cite{Ruij}, we have the identities
\begin{gather}
\Gamma_1(w\,|\,a) = \frac{a^{w/a-1/2}}{\sqrt{2\pi}} \,\Gamma(w/a), \label{gamma1}\\
\zeta_1(s, w\,|\,a) = a^{-s}\zeta(s,\,w/a).
\end{gather}

We now proceed to state our results. The proofs are deferred to Section 8 for reader's convenience.
\begin{theorem}[Scaling invariance]\label{scaling}
Let \(\Re(w)>0,\) \(\kappa>0\) and \((\kappa\,a)_i\triangleq\kappa\,a_i,\;i=1\cdots M.\)
\begin{equation}\label{scale}
\Gamma_M(\kappa w\,|\,\kappa a) = \kappa^{-B_{M,M}(w\,|\,a)/M!}\,\Gamma_M(w\,|\,a).
\end{equation}
\end{theorem}
\begin{remark}
In the case of classical normalization, this result appears to be due to \cite{KataOhts}.
It was re-discovered in \cite{Kuz} in the special case of $M=2.$
\end{remark}
\begin{theorem}[Barnes multiplication]\label{multiplic}
Let $\Re(w)>0$ and $k=1,2,3,\cdots.$
\begin{equation}
\Gamma_M(kw\,|\,a) = k^{-B_{M,
M}(kw\,|\,a)/M!}\,\prod\limits_{p_1,\cdots,p_M=0}^{k-1}\Gamma_M\Bigl(w+\frac{\sum_{j=1}^M
p_j a_j}{k}\,|\,a\Bigr).
\end{equation}
\end{theorem}
\begin{remark}
In the classical case, this result is due to \cite{mBarnes}.
\end{remark}
\begin{theorem}[Shintani factorization]\label{ShinFactor}
Let $g(t)$ be a function of Ruijsenaars class. Define a sequence of functions \(f(t)\) that are parameterized by \(M\in\mathbb{N}\) and $a_j>0$
\begin{equation}\label{fgdef}
f(t)\triangleq  \frac{t^M}{\prod\limits_{j=1}^M 1-e^{-a_j t}}\,g(t)
\end{equation}
and denote the corresponding gamma function by \(\Gamma^{(f)}_M(w\,|\,a)\) and Bernoulli polynomials by \(B^{(f)}_{M,m}(x\,|\,a),\) \(a=(a_1\cdots a_M).\)
Let the function \(\Psi^{(f)}_{M+1}(x,y\,|\,a)\) 
be defined for \(\Re(x),\Re(y)>0\) by
\begin{align}
\Psi^{(f)}_{M+1}(x,y\,|\,a) & \triangleq \int\limits_0^\infty \frac{dt}{t^{M+1}}\Bigl[\sum\limits_{m=0}^M \frac{t^m}{m!}B^{(f)}_{M,m}(x\,|\,a)e^{-yt} - \sum\limits_{m=0}^{M-1}\frac{t^m}{m!}B^{(f)}_{M,m}(x+y\,|\,a) - \nonumber \\ & - \frac{t^M}{M!}B^{(f)}_{M,M}(x+y\,|\,a)e^{-t}\Bigr]+\frac{1}{y(M+1)!}B^{(f)}_{M,M+1}(x\,|\,a). \label{Psidef}
\end{align}
Given arbitrary $x>0,$ there exists a function $\phi^{(f)}_{M+1}\bigl(w,x\,|\,a, a_{M+1}\bigr)$ such that
\begin{align}
\Gamma^{(f)}_{M+1}\bigl(w\,|\,a,a_{M+1}\bigr) = & \prod\limits_{k=1}^\infty \frac{\Gamma^{(f)}_M(w+ka_{M+1}\,|\,a)}{\Gamma^{(f)}_M(x+ka_{M+1}\,|\,a)}e^{\Psi^{(f)}_{M+1}(x,ka_{M+1}\,|\,a)-\Psi^{(f)}_{M+1}(w,ka_{M+1}\,|\,a)} \star \nonumber \\
&\star\exp{\bigl(\phi^{(f)}_{M+1}(w,x\,|\,a, a_{M+1})\bigr)}\, \Gamma^{(f)}_{M}(w\,|\,a). \label{generalfactorization}
\end{align}
\(\Psi^{(f)}_{M+1}(w,y\,|\,a)\) and \(\phi^{(f)}_{M+1}(w,x\,|\,a, a_{M+1})\) are polynomials in $w$ of degree $M+1.$
\end{theorem}
\begin{remark}
The case of $g(t)=1,$ which corresponds to Barnes multiple gamma functions, 
was first treated using classical normalization in \cite{Shintani} for $M=1$ and in \cite{KataOhts} in general. We gave a new proof for $M=1$  along the lines of the approach taken in this paper (using the classical normalization) in \cite{Me}. In all of these papers the variable $x$ was taken  to be $a_1$. The equivalence of \eqref{Psidef} and the definition of \(\Psi^{(f)}_{M+1}(x,y\,|\,a)\) in \cite{KataOhts} is shown in Section 8, where we also give an explicit formula for the function \(\phi^{(f)}_{M+1}(w,x\,|\,a,a_{M+1}).\)
\end{remark}

\section{A Review of Barnes Beta Distributions}
In this section we will give a review of key results on the existence and properties of Barnes beta distributions that we established in \cite{Me13}.
Let $\lbrace b_k\rbrace,$ $k\in\mathbb{N}$ be a sequence of positive
real numbers and $N, M\in\mathbb{N}.$ Let the symbol $\sum_{k_1<\cdots<k_p=1}^N$ denote the sum
over all indices $k_i=1\cdots N,$ $i=1\cdots p,$ satisfying
$k_1<\cdots<k_p.$ Define the action of the combinatorial operator $\mathcal{S}_N$ on a function \(h(x)\) by
\begin{definition}\label{Soperator}
\begin{equation}\label{S}
(\mathcal{S}_Nh)(q\,|\,b) \triangleq \sum\limits_{p=0}^N (-1)^p
\sum\limits_{k_1<\cdots<k_p=1}^N
h\bigl(q+b_0+b_{k_1}+\cdots+b_{k_p}\bigr).
\end{equation}
\end{definition}
\noindent In other words, in \eqref{S} the action of $\mathcal{S}_N$
is defined as an alternating sum over all combinations of $p$
elements for every $p=0\cdots N.$ Given a function $f(t)$ of
Ruijsenaars class, confer Section 2, such that $f(t)>0$ for $t\geq
0,$ let $L^{(f)}_M(w)$ be the corresponding generalized log-gamma function
defined in \eqref{Ldef}. The main example is the function $f(t)$ in
\eqref{fdef} so that $L^{(f)}_M(w)=L_M(w\,|\,a)$ is the Barnes multiple
log-gamma function.
\begin{definition}\label{bdef}
Given $q\in\mathbb{C}-(-\infty, -b_0],$ let
\begin{equation}\label{eta}
\eta^{(f)}_{M,N}(q\,|\,b) \triangleq \exp\Bigl(\bigl(\mathcal{S}_N
L^{(f)}_M\bigr)(q\,|\,b) - \bigl(\mathcal{S}_N L^{(f)}_M\bigr)(0\,|\,b)\Bigr).
\end{equation}
\end{definition}
The function $\eta^{(f)}_{M,N}(q\,|\,b)$ is holomorphic over
$q\in\mathbb{C}-(-\infty, -b_0]$ and equals a product of ratios of
generalized gamma functions by construction.
\begin{theorem}[Existence]\label{main}
Given $M, N\in\mathbb{N}$ such that $M\leq N,$ the function
$\eta^{(f)}_{M,N}(q\,|\,b)$ is the Mellin transform of a probability
distribution on $(0, 1].$ Denote it by $\beta^{(f)}_{M, N}(b).$ Then,
\begin{equation}
{\bf E}\bigl[\beta^{(f)}_{M, N}(b)^q\bigr] = \eta^{(f)}_{M, N}(q\,|\,b),\;
\Re(q)>-b_0.
\end{equation}
The distribution $-\log\beta^{(f)}_{M, N}(b)$ is infinitely divisible on
$[0, \infty)$ and has the L\'evy-Khinchine decomposition for $\Re(q)<b_0,$
\begin{equation}\label{LKH}
{\bf E}\Bigl[\exp\bigl(-q\log\beta^{(f)}_{M, N}(b)\bigr)\Bigr] =
\exp\Bigl(\int\limits_0^\infty (e^{tq}-1) e^{-b_0
t}\prod\limits_{j=1}^N (1-e^{-b_j t}) \frac{f(t)}{t^{M+1}} dt\Bigr).
\end{equation}
\end{theorem}
\begin{corollary}[Structure]\label{Structure}
$\log\beta^{(f)}_{M,N}(b)$ is absolutely continuous if and only if $M=N.$ If $M<N,$
$-\log\beta^{(f)}_{M,N}(b)$ is compound Poisson and
\begin{subequations}
\begin{align}
{\bf P}\bigl[\beta^{(f)}_{M,N}(b)=1\bigr] & =
\exp\Bigl(-\int\limits_0^\infty e^{-b_0 t}\prod\limits_{j=1}^N
(1-e^{-b_j t}) \frac{f(t)}{t^{M+1}} dt\Bigr), \label{Pof1} \\
& = \exp\bigl(-(\mathcal{S}_N L^{(f)}_M)(0\,|\,b)\bigr). \label{Pof12}
\end{align}
\end{subequations}
\end{corollary}

From now on we restrict our attention to Barnes multiple gamma
functions, \emph{i.e.} $f(t)$ is as in \eqref{fdef}, and write
$\eta_{M, N}(q\,|\,a,\,b)$ to indicate dependence on $(a_1,\cdots,
a_M)$ and $(b_0,\cdots, b_N).$ Also, $\hat{c}_i \triangleq (\cdots, c_{i-1},\,c_{i+1},\cdots),$ and
$b_j+x \triangleq \cdots b_{j-1}, b_j+x, b_{j+1}, \cdots.$
Note that $\eta_{M,
N}(q\,|\,a,\,b)$ is symmetric in $(a_1,\cdots, a_M)$ and
$(b_1,\cdots, b_N).$
\begin{theorem}[Functional equation]\label{FunctEquat}
$1\leq M\leq N,$ $q\in\mathbb{C}-(-\infty, -b_0],$ $i=1\cdots M,$
\begin{equation}
\eta_{M, N}(q+a_i\,|\,a,\,b) =
\eta_{M, N}(q\,|\,a,\,b)\,\exp\bigl(-(\mathcal{S}_N
L_{M-1})(q\,|\,\hat{a}_i, b)\bigr). \label{fe1}
\end{equation}
\end{theorem}
\begin{corollary}[Barnes beta algebra]\label{algebra}
$1\leq M\leq N,$ $i=1\cdots M,$ $j=1\cdots N,$
\begin{align}
\beta_{M, N-1}(a,\,\hat{b}_j) &\overset{{\rm in \,law}}{=}\beta_{M,
N}(a,\,b)\,\beta_{M, N-1}(a,\,b_0+b_j,\,\hat{b}_j), \label{algebra1} \\
\beta_{M, N}(a,\,b) &\overset{{\rm in \,law}}{=}\beta_{M,
N}(a,\,b_0+a_i)\,\beta_{M-1, N}(\hat{a}_i,\,b), \label{algebra2} \\
\beta_{M, N}(a,\,b_j+a_i) &\overset{{\rm in \,law}}{=}\beta_{M,
N}(a,\,b)\,\beta_{M-1, N-1}(\hat{a}_i,\,b_0+b_j,\hat{b}_j), \label{algebra3} \\
\beta_{M, N}(a,\,b_j+a_i) &\overset{{\rm in \,law}}{=}\beta_{M,
N}(a,\,b_0+a_i)\,\beta_{M-1, N-1}(\hat{a}_i,\,\hat{b}_j). \label{algebra4}
\end{align}
\end{corollary}
\begin{theorem}[Shintani factorization]\label{Factorization}
Given $1\leq M\leq N$ and $q\in\mathbb{C}-(-\infty, -b_0],$
\begin{equation}\label{infinprod2}
\eta_{M,N}(q\,|\,a, b) = \prod\limits_{k=0}^\infty \frac{\eta_{M-1,N}(q+k
a_i\,|\,\hat{a}_i, b)}{\eta_{M-1,N}(k
a_i\,|\,\hat{a}_i, b)}.
\end{equation}
\end{theorem}
\begin{corollary}[Solution to the functional equation]\label{solution}
The infinite product representation in Theorem \ref{Factorization}
is the solution to the functional equation in Theorem
\ref{FunctEquat}. Probabilistically,
\begin{equation}\label{funcequatsolut}
\beta_{M, N}(a,\,b) \overset{{\rm in \,law}}{=}
\prod\limits_{k=0}^\infty \beta_{M-1, N}(\hat{a}_i,\,b_0+ka_i).
\end{equation}
\end{corollary}
\begin{theorem}[Moments]\label{IntegralMoments} Let $k\in\mathbb{N}.$
\begin{align}
{\bf E}\bigl[\beta_{M, N}(a, b)^{k a_i}\bigr] &  =
\exp\Bigl(-\sum\limits_{l=0}^{k-1} \bigl(\mathcal{S}_N
L_{M-1}\bigr)(l a_i\,|\,\hat{a}_i, b)\Bigr), \label{posmom} \\
{\bf E}\bigl[\beta_{M, N}(a, b)^{-k a_i}\bigr] & =
\exp\Bigl(\sum\limits_{l=0}^{k-1} \bigl(\mathcal{S}_N
L_{M-1}\bigr)(-(l+1) a_i\,|\,\hat{a}_i, b)\Bigr), \; k a_i<b_0. \label{negmom}
\end{align}
\end{theorem}

We have two results that hold for special values of $a$ and $b.$
\begin{theorem}[Reduction to independent factors]\label{Reduction}
Given $i, j,$ if $b_j=n\,a_i$ for $n\in\mathbb{N},$
\begin{equation}\label{reductequat}
\beta_{M, N}(a, b) \overset{{\rm in \,law}}{=} \prod_{k=0}^{n-1}
\beta_{M-1, N-1}\bigl(\hat{a}_i, b_0+ka_i,\, \hat{b}_j\bigr).
\end{equation}
\end{theorem}
\begin{theorem}[Moments]\label{Momentsaspecial} Let $a_i=1$ for all $i=1\cdots M.$ Then, for any
$n\in\mathbb{N},$
\begin{align}
{\bf E}\bigl[\beta_{M, N}(a, b)^n\bigr]  & =
\prod\limits_{i=1}^{M-1}e^{(-1)^i \binom{n}{i} (\mathcal{S}_N
L_{M-i})(0\,|\,b)} \prod\limits_{i_1=0}^{n-1}
\prod\limits_{i_2=0}^{i_1-1} \cdots \prod\limits_{i_M=0}^{i_{M-1}-1}
e^{(-1)^M (\mathcal{S}_N L_{0})(i_M\,|\,b)}, \nonumber \\
& = \prod\limits_{i=1}^{M-1}e^{(-1)^i \binom{n}{i} (\mathcal{S}_N
L_{M-i})(0\,|\,b)} \prod\limits_{i_1=0}^{n-1}
\prod\limits_{i_2=0}^{i_1-1} \cdots \prod\limits_{i_M=0}^{i_{M-1}-1}
\Bigl[\frac{\prod\limits_{j_1=1}^N (i_M+b_0+b_{j_1})}{(i_M+b_0)} \star \nonumber \\
& \star \frac{\prod\limits_{j_1<j_2<j_3}^N
(i_M+b_0+b_{j_1}+b_{j_2}+b_{j_3})}{\prod\limits_{j_1<j_2}^N
(i_M+b_0+b_{j_1}+b_{j_2})}\cdots\Bigr]^{(-1)^M}.
\end{align}
\end{theorem}
\begin{remark}
The generating function of the moments, \emph{i.e.} the Laplace
transform (known to have the infinite radius of convergence), in the case of $a_i=1$ for all $i=1\cdots M$ gives an
interesting generalization of the confluent hypergeometric function
corresponding to $M=N=1.$ It is an interesting open question to derive an ODE that the Laplace transform must satisfy in the general case. 
\end{remark}

We note that the structure of $\beta_{M,N}(a, b)$ depends on
rationality of $(a_1,\cdots, a_M)$ and $(b_1,\cdots, b_M).$ This is
clear from Definition \ref{bdef} as this structure is determined by
ratios of multiple gamma functions that have poles specified in
\eqref{poles}, confer also Theorem \ref{BarnesFactorization}. 
This phenomenon was studied in a different context
for the double gamma function in \cite{HacKuz}, \cite{HubKuz}, and \cite{Kuz}.

We conclude this section with the case of $M=N=2$ in order to illustrate the general theory with a non-trivial example. In addition, this case is also of a particular interest in the probabilistic theory of the Selberg integral.

Let $a_1=1$ and $a_2=\tau>0$ and write $\beta_{2, 2}(\tau, b),$
$\eta_{2,2}(q\,|\,\tau, b),$ and
$\Gamma_2(w\,|\,\tau)$ for
brevity. From Definition \ref{bdef} and Theorem \ref{main} we have
${\bf E}\bigl[\beta_{2, 2}(\tau, b)^q\bigr] = \eta_{2,2}(q\,|\,\tau,
b)$ for $\Re(q)>-b_0$ and $\eta_{2,2}(q\,|\,\tau, b)$ is given by
\begin{equation}\label{beta22}
\frac{\Gamma_2(q+b_0\,|\,\tau)}{\Gamma_2(b_0\,|\,\tau)}
\frac{\Gamma_2(b_0+b_1\,|\,\tau)}{\Gamma_2(q+b_0+b_1\,|\,\tau)}
\frac{\Gamma_2(b_0+b_2\,|\,\tau)}{\Gamma_2(q+b_0+b_2\,|\,\tau)}
\frac{\Gamma_2(q+b_0+b_1+b_2\,|\,\tau)}{\Gamma_2(b_0+b_1+b_2\,|\,\tau)}.
\end{equation}
Using \eqref{gamma1}, the functional equation in Theorem
\ref{FunctEquat} takes the form
\begin{align}
\eta_{2, 2}(q+1\,|\,\tau, b) & = \eta_{2, 2}(q\,|\,\tau, b)\,
\frac{\Gamma\bigl((q+b_0+b_1)/\tau\bigr)\Gamma\bigl((q+b_0+b_2)/\tau\bigr)}
{\Gamma\bigl((q+b_0)/\tau\bigr)\Gamma\bigl((q+b_0+b_1+b_2)/\tau\bigr)},
\\
\eta_{2, 2}(q+\tau\,|\,\tau, b) & = \eta_{2, 2}(q\,|\,\tau, b)\,
\frac{\Gamma(q+b_0+b_1)\Gamma(q+b_0+b_2)}
{\Gamma(q+b_0)\Gamma\bigl(q+b_0+b_1+b_2)}.
\end{align}
The factorization equations in Theorem \ref{Factorization} are
\begin{align} 
\eta_{2,2}(q\,|\,\tau, b) &  =
 \prod\limits_{k=0}^\infty\Bigl[
\frac{\Gamma((q+k+b_0)/\tau) }{\Gamma((k+b_0)/\tau)}
\frac{\Gamma((k+b_0+b_1)/\tau)}{\Gamma((q+k+b_0+b_1)/\tau)}\star
\nonumber \\ & \star
\frac{\Gamma((k+b_0+b_2)/\tau)}{\Gamma((q+k+b_0+b_2)/\tau)}
\frac{\Gamma((q+k+b_0+b_1+b_2)/\tau)}{\Gamma((k+b_0+b_1+b_2)/\tau)}\Bigr],
\\
\eta_{2,2}(q\,|\,\tau, b) & = \prod\limits_{k=0}^\infty\Bigl[
\frac{\Gamma(q+k\tau+b_0)}{\Gamma(k\tau+b_0)}
\frac{\Gamma(k\tau+b_0+b_1)}{\Gamma(q+k\tau+b_0+b_1)}\star
\nonumber \\ & \star
\frac{\Gamma(k\tau+b_0+b_2)}{\Gamma(q+k\tau+b_0+b_2)}
\frac{\Gamma(q+k\tau+b_0+b_1+b_2)}{\Gamma(k\tau+b_0+b_1+b_2)}\Bigr].\label{eta22infinfactor}
\end{align}
The integral moments in Theorem \ref{IntegralMoments} are
\begin{align}
{\bf E}\bigl[\beta_{2, 2}(\tau, b)^k\bigr] & =
\prod\limits_{l=0}^{k-1}
\Bigl[\frac{\Gamma\bigl((l+b_0+b_1)/\tau\bigr)\,\Gamma\bigl((l+b_0+b_2)/\tau\bigr)}{\Gamma\bigl((l
+b_0)/\tau\bigr)\,\Gamma\bigl((l+b_0+b_1+b_2)/\tau\bigr)}\Bigr], \\
{\bf E}\bigl[\beta_{2, 2}(\tau, b)^{-k}\bigr] & =
\prod\limits_{l=0}^{k-1} \Bigl[\frac{\Gamma\bigl((-(l+1)
+b_0)/\tau\bigr)\,\Gamma\bigl((-(l+1)+b_0+b_1+b_2)/\tau\bigr)}{\Gamma\bigl((-(l+1)+b_0+b_1)/\tau\bigr)\,
\Gamma\bigl((-(l+1)+b_0+b_2)/\tau\bigr)}\Bigr],
\end{align}
for $k<b_0.$ Finally, the positive integral moments in case of $\tau=1$ take on a generalized hypergeometric form by Theorem \ref{Momentsaspecial}.
Let $(b)_i\triangleq b(b+1)\cdots (b+i-1),$ then
\begin{equation}
{\bf E}\bigl[\beta_{2, 2}(\tau=1, b)^k\bigr] =
\Bigl[\frac{\Gamma(b_0+b_1)\Gamma(b_0+b_2)}{\Gamma(b_0)\Gamma(b_0+b_1+b_2)}\Bigr]^k
\prod\limits_{i=0}^{k-1} \frac{(b_0+b_1)_i (b_0+b_2)_i}{(b_0)_i
(b_0+b_1+b_2)_i}.
\end{equation}

\section{New Results on Barnes Beta Distributions}
In this section we will present our new results on general Barnes beta distributions $\beta_{M,N}(a, b).$ In Theorem \ref{barnesbetascaling} we establish the scaling invariance property of $\beta_{M,N}(a, b)$ that is based on Theorem \ref{scaling}. In Theorem \ref{NewShinProof} we give a new proof of the Shintani factorization stated in Theorem \ref{Factorization} that is based on Theorem \ref{ShinFactor}. In Theorem \ref{BarnesFactorization} we derive the Barnes factorization of the Mellin transform of $\beta_{M,N}(a, b)$ as a corollary of the Shintani factorization. In Theorem \ref{SLaction} we show that
a combination of the Barnes factorization and the functional equation of the Mellin transform stated in Theorem \ref{FunctEquat} gives an explicit formula for the action of $\mathcal{S}_N$ on the log-gamma function $L_M.$ This formula
implies Corollaries \ref{newmoments} and \ref{inequal1}, in which we give explicit formulas for the integral moments and mass at 1 of $\beta_{M,N}(a, b).$  Finally, in Corollary \ref{BarnesFactorSpecial} we treat the Barnes factorization in the special case of $a_i=1.$ We begin by reminding the reader of the following elementary observation that we made in \cite{Me13}. 
\begin{lemma}[Action of $\mathcal{S}_N$ on polynomials]\label{MyLemma}
Given $n=0\cdots N$ and arbitrary $q,$
\begin{subequations}\label{Sonpower}
\begin{align}
\bigl(\mathcal{S}_N x^n\bigr)(q\,|\,b) & = 0,\; n<N, \\
\bigl(\mathcal{S}_N x^n\bigr)(q\,|\,b) & = \bigl(\mathcal{S}_N x^n\bigr)(0\,|\,b), \;n=N.
\end{align}
\end{subequations}
\end{lemma}
\begin{theorem}[Scaling invariance]\label{barnesbetascaling}
Let $\kappa>0.$ Then,
\begin{equation}
\beta^{\kappa}_{M, N}(\kappa\,a, \kappa\,b) \overset{{\rm in \,law}}{=}\beta_{M,
N}(a,\,b).
\end{equation}
\end{theorem}
\begin{proof}
By Theorem \ref{scaling} and the definition of the operator $\mathcal{S}_N$ in \eqref{Soperator}, we have the identity
\begin{equation}
\bigl(\mathcal{S}_N L_M\bigr)(\kappa\,q\,|\,\kappa\,a,\kappa\,b) = \bigl(\mathcal{S}_N L_M\bigr)(q\,|\,a, b) - \frac{\log\kappa}{M!}
\bigl(\mathcal{S}_N B_{M,M}\bigr)(q\,|\,b)
\end{equation}
so that by Lemma \ref{MyLemma} we obtain for $M\leq N$
\begin{align}
\log\eta_{M, N}(\kappa\,q\,|\,\kappa\,a,\kappa\,b) & = \bigl(\mathcal{S}_N L_M\bigr)(\kappa\,q\,|\,\kappa\,a,\kappa\,b) - \bigl(\mathcal{S}_N L_M\bigr)(0\,|\,\kappa\,a,\kappa\,b), \nonumber \\
& = \log\eta_{M, N}(q\,|\,a,\,b).
\end{align}
The result follows.
\end{proof}
\begin{theorem}[Shintani factorization]\label{NewShinProof}
The infinite factorization of the Mellin transform in Theorem \ref{Factorization} is a corollary of Theorem \ref{ShinFactor}.
\end{theorem}
\begin{proof}
By Theorem \ref{ShinFactor}, Lemma \ref{MyLemma}, and the fact that \(\Psi_{M}(w,y\,|\,a)\) and \(\phi_{M}(w,x\,|a)\) are polynomials in $w$ of degree $M$ and $M\leq N,$ we can write
\begin{align}
\bigl(\mathcal{S}_N L_M\bigr)(q\,|\,a, b) - \bigl(\mathcal{S}_N L_M\bigr)(0\,|\,a, b) & = \sum\limits_{k=0}^\infty \Bigl[\bigl(\mathcal{S}_N L_{M-1}\bigr)(q+ka_M\,|\,a, b) - \nonumber \\
& -\bigl(\mathcal{S}_N L_{M-1}\bigr)(ka_M\,|\,a, b)\Bigr],
\end{align}
which is equivalent to \eqref{infinprod2}.
\end{proof}
\begin{remark}
The original proof of Theorem \ref{Factorization} in \cite{Me13} used a special
property of the L\'evy-Khinchine representation given in Theorem \ref{main}
to verify \eqref{funcequatsolut} directly, which is equivalent to \eqref{infinprod2} by the following identity that we noted in \cite{Me13}.
\begin{equation}\label{fe2}
\eta_{M, N}(q\,|\,a,\,b_0+x)\,\eta_{M, N}(x\,|\,a,\,b) = \eta_{M, N}(q+x\,|\,a,\,b), \; x>0.
\end{equation}
The new proof explains that the origin  of the infinite product factorization of the Mellin transform is the Shintani factorization of the multiple gamma function.
\end{remark}

Next, we proceed to state and prove a new infinite product representation of the Mellin transform of $\beta_{M,N}(a,b)$ that corresponds to the classical infinite product representation of the multiple gamma function given by Barnes
in \cite{mBarnes}.
\begin{theorem}[Barnes factorization]\label{BarnesFactorization}
Given $0\leq M\leq N$ and $q\in\mathbb{C}-(-\infty, -b_0],$
\begin{align}
\eta_{M,N}(q\,|\,a, b) & = \prod\limits_{n_1,\cdots ,n_M=0}^\infty \Bigl[
\frac{b_0+\Omega}{q+b_0+\Omega}\prod\limits_{j_1=1}^N \frac{q+b_0+b_{j_1}+\Omega}
{b_0+b_{j_1}+\Omega} \star \nonumber \\
& \star \prod\limits_{j_1<j_2}^N \frac{b_0+b_{j_1}+b_{j_2}+\Omega}
{q+b_0+b_{j_1}+b_{j_2}+\Omega}
\prod\limits_{j_1<j_2<j_3}^N \frac{q+b_0+b_{j_1}+b_{j_2}+b_{j_3}+\Omega}{b_0+b_{j_1}+b_{j_2}+b_{j_3}+\Omega}\cdots\Bigr], \label{barnesfactor}
\end{align}
where $\Omega\triangleq \sum_{i=1}^M n_i \, a_i$ and the last ratio in the product is $\bigl((q+b_0+\sum_{j=1}^N b_j+\Omega)/(b_0+\sum_{j=1}^N b_j+\Omega)\bigr)^{(-1)^{N-1}}.$ Probabilistically,
\begin{equation}\label{probfuncequatsolut}
\beta_{M, N}(a,\,b) \overset{{\rm in \,law}}{=}
\prod\limits_{n_1,\cdots ,n_M=0}^\infty \beta_{0, N}(b_0+\Omega).
\end{equation}
\end{theorem}
\begin{proof}
We will give proof by induction on $M$ using Theorem \ref{Factorization}.
Let $N$ be fixed. If $M=0,$ \eqref{barnesfactor} is true by Definition \ref{bdef} and $\Gamma_0(w)=1/w.$ Assume \eqref{barnesfactor} is true for
$M-1.$ Substituting \eqref{barnesfactor} for $M-1$ into \eqref{infinprod2}, writing $n_M$ for the index $k$ in \eqref{infinprod2}, and denoting $\widetilde{\Omega}\triangleq  \sum_{i=1}^{M-1} n_i \, a_i,$ we obtain for
$\eta_{M,N}(q\,|\,a, b)$
\begin{equation}
\prod\limits_{n_1,\cdots ,n_M=0}^\infty \Bigl[
\frac{{\displaystyle\frac{b_0+\widetilde{\Omega}}{q+b_0+\widetilde{\Omega}+n_M\,a_M}\prod\limits_{j_1=1}^N \frac{q+b_0+b_{j_1}+\widetilde{\Omega}+n_M\,a_M}{b_0+b_{j_1}+\widetilde{\Omega}}\cdots}}
{{\displaystyle\frac{b_0+\widetilde{\Omega}}{b_0+\widetilde{\Omega}+n_M\,a_M}
\prod\limits_{j_1=1}^N \frac{b_0+b_{j_1}+\widetilde{\Omega}+n_M\,a_M}{b_0+b_{j_1}+\widetilde{\Omega}}\cdots}}\Bigr],
\end{equation}
which equals \eqref{barnesfactor} after obvious cancellations as $\Omega=\widetilde{\Omega}+n_M\,a_M.$ To prove \eqref{probfuncequatsolut}
it is sufficient to note that \eqref{barnesfactor} is equivalent to
\begin{equation}
\eta_{M,N}(q\,|\,a, b) = \prod\limits_{n_1,\cdots ,n_M=0}^\infty \eta_{0,N}(q\,|\, b_0+\Omega),
\end{equation}
which in turn is equivalent to \eqref{probfuncequatsolut}. 
Alternatively, \eqref{probfuncequatsolut} follows from \eqref{funcequatsolut} directly by induction.
\end{proof}
\begin{theorem}[Action of $\mathcal{S}_N$ on $L_M$]\label{SLaction}
Let $M<N$ and $\Omega\triangleq \sum_{i=1}^M n_i \, a_i.$ Then,
\begin{align}
\exp\Bigl(-\bigl(\mathcal{S}_N L_M\bigr)(q\,|\,a,b)\Bigr) & = \prod\limits_{n_1,\cdots ,n_M=0}^\infty \Bigl[
\frac{q+b_0+\Omega}{\prod\limits_{j_1=1}^N q+b_0+b_{j_1}+\Omega}
\star \nonumber \\
& \star\frac{\prod\limits_{j_1<j_2}^N q+b_0+b_{j_1}+b_{j_2}+\Omega}
{\prod\limits_{j_1<j_2<j_3}^N q+b_0+b_{j_1}+b_{j_2}+b_{j_3}+\Omega}\cdots\Bigr], 
\end{align}
where the last factor is $\bigl(q+b_0+\sum_{j=1}^N b_j+\Omega\bigr)^{(-1)^N}.$
\end{theorem}
\begin{proof}
This is a corollary of Theorems \ref{FunctEquat} and \ref{BarnesFactorization}.
Let $M\leq N$ and $\widetilde{\Omega}\triangleq  \sum_{i=1}^{M-1} n_i \, a_i.$
By \eqref{barnesfactor} we get for $\eta_{M,N}(q+a_M\,|\,a, b)/\eta_{M,N}(q\,|\,a, b)$ 
\begin{equation}
\prod\limits_{n_1,\cdots ,n_{M}=0}^\infty \Bigl[
\frac{q+b_0+\widetilde{\Omega}+n_M\,a_M}{q+b_0+\widetilde{\Omega}+n_M\,a_M+a_M}\prod\limits_{j_1=1}^N \frac{q+b_0+b_{j_1}+\widetilde{\Omega}+n_M\,a_M+a_M}{q+b_0+b_{j_1}+\widetilde{\Omega}+n_M\,a_M} \cdots\Bigr].
\end{equation}
Now, the product over $n_M$ simplifies to
\begin{align}
& \prod\limits_{n_M=0}^\infty  \Bigl[
\frac{q+b_0+\widetilde{\Omega}+n_M\,a_M}{q+b_0+\widetilde{\Omega}+n_M\,a_M+a_M}\prod\limits_{j_1=1}^N \frac{q+b_0+b_{j_1}+\widetilde{\Omega}+n_M\,a_M+a_M}{q+b_0+b_{j_1}+\widetilde{\Omega}+n_M\,a_M} \cdots\Bigr] = \nonumber \\
& \frac{q+b_0+\widetilde{\Omega}}{\prod\limits_{j_1=1}^N q+b_0+b_{j_1}+\widetilde{\Omega}}
\frac{\prod\limits_{j_1<j_2}^N q+b_0+b_{j_1}+b_{j_2}+\widetilde{\Omega}}
{\prod\limits_{j_1<j_2<j_3}^N q+b_0+b_{j_1}+b_{j_2}+b_{j_3}+\widetilde{\Omega}}\cdots
\end{align}
The result now follows by \eqref{fe1} and relabeling $M-1\rightarrow M.$
\end{proof}
\begin{remark}
Barnes \cite{mBarnes} gave an infinite product formula for $\Gamma_M(w\,|\,a),$
which in our normalization of multiple gamma functions takes on the form
\begin{equation}
\Gamma^{-1}_M(w\,|\,a) = e^{P(w\,|\,a)}\,w\prod\limits_{n_1,\cdots , n_M=0}^\infty
{}' \Bigl(1+\frac{w}{\Omega}\Bigr)\exp\Bigl(\sum_{k=1}^M \frac{(-1)^k}{k}\frac{w^k}{\Omega^k}\Bigr), 
\end{equation}
where $P(w\,|\,a)$ is a polynomial in $w$ of degree $M,$ $\Omega$ is as in Theorem \ref{BarnesFactorization}, and the prime indicates
that the product is over all indices except $n_1=\cdots =n_M=0.$ Given this result, Theorems \ref{BarnesFactorization} and \ref{SLaction} follow
by Lemma \ref{MyLemma}. 
\end{remark}

Theorem \ref{SLaction} implies new explicit formulas for the moments and mass at 1. 
\begin{corollary}[Moments]\label{newmoments} Let $k\in\mathbb{N},$ $M\leq N,$ and 
$\Omega\triangleq\sum_{i=1}^{M-1} n_i \, a_i.$
\begin{align}
{\bf E}\bigl[\beta_{M, N}(a, b)^{k a_M}\bigr] &  =
\prod\limits_{l=0}^{k-1}\prod\limits_{n_1,\cdots ,n_{M-1}=0}^\infty \Bigl[
\frac{l\,a_M+b_0+\Omega}{\prod\limits_{j_1=1}^N l\,a_M+b_0+b_{j_1}+\Omega}
\star \nonumber \\
& \star\frac{\prod\limits_{j_1<j_2}^N l\,a_M+b_0+b_{j_1}+b_{j_2}+\Omega}
{\prod\limits_{j_1<j_2<j_3}^N l\,a_M+b_0+b_{j_1}+b_{j_2}+b_{j_3}+\Omega}\cdots\Bigr], \label{posmomentsBarnes}
\end{align} 
\end{corollary}
\begin{proof}
This is an immediate corollary of \eqref{posmom} in Theorem \ref{IntegralMoments}
and the formula for $\exp\bigl(-(\mathcal{S}_N L_M)(q\,|\,a,b)\bigr)$ in Theorem \ref{SLaction}.
\end{proof}
The formula for the negative moments is obtained in the same way by \eqref{negmom}.
\begin{corollary}[Mass at 1]\label{inequal1} Let $M<N$ and $\Omega\triangleq\sum_{i=1}^{M} n_i \, a_i.$
\begin{equation}
{\bf P}\bigl[\beta_{M,N}(a,b)=1\bigr] = 
\prod\limits_{n_1,\cdots ,n_M=0}^\infty \Bigl[
\frac{b_0+\Omega}{\prod\limits_{j_1=1}^N b_0+b_{j_1}+\Omega}
\frac{\prod\limits_{j_1<j_2}^N b_0+b_{j_1}+b_{j_2}+\Omega}
{\prod\limits_{j_1<j_2<j_3}^N b_0+b_{j_1}+b_{j_2}+b_{j_3}+\Omega}\cdots\Bigr].
\end{equation}
\end{corollary}
\begin{proof}
This is an immediate corollary of \eqref{Pof12} in Corollary \ref{Structure} and
Theorem \ref{SLaction}.
\end{proof}

We will conclude this section with the special case of Theorem \ref{BarnesFactorization} that corresponds to $a_i=1$ for all $i=1\cdots M$
and appears in Section 6 below. 
\begin{corollary}[Barnes factorization for $a_i=1$]\label{BarnesFactorSpecial} Let $0\leq M\leq N$ and $a_i=1$ for all $i=1\cdots M.$ Then, 
\begin{align}
\eta_{M,N}(q\,|\,1, b) & = \prod\limits_{k=0}^\infty \Bigl[
\frac{b_0+k}{q+b_0+k}\prod\limits_{j_1=1}^N \frac{q+b_0+b_{j_1}+k}
{b_0+b_{j_1}+k} \prod\limits_{j_1<j_2}^N \frac{b_0+b_{j_1}+b_{j_2}+k}
{q+b_0+b_{j_1}+b_{j_2}+k} \star \nonumber \\
& \star \prod\limits_{j_1<j_2<j_3}^N \frac{q+b_0+b_{j_1}+b_{j_2}+b_{j_3}+k}{b_0+b_{j_1}+b_{j_2}+b_{j_3}+k}\cdots\Bigr]^{(k\,|\,M)},
\label{barnesfactorspecial} \\
(k\,|\,M) & \triangleq \sum\limits_{m=1}^M \binom{k-1}{m-1}\binom{M}{m}.
\end{align}
\end{corollary}
\begin{proof}
This follows from \eqref{barnesfactor} using the formula for the number of compositions of $k$ into exactly $m$ parts, confer \cite{Bressoud}, \emph{i.e.} $(k\,|\,M)$ equals the number of non-negative integer solutions to $n_1+\cdots +n_M=k.$
\end{proof}
\begin{remark} The analytic structure of $\eta_{M,N}(a,b)$ is not fully understood, even in the case of $M=0.$ The latter with $b_i=i,$ $i=1\cdots N,$ is of particular interest as it occurs in the context of the Riemann xi function, confer Remark \ref{Splitting} below.
\end{remark}
\section{New Results on Selberg Integral}

In this section we will consider the case of $M=N=2$ of Barnes beta distributions and state and prove a ``master'' theorem (Theorem \ref{general}), which gives a new, purely probabilistic proof of the key result of \cite{Me} on the extension of the Selberg integral as a function of its dimension to the Mellin transform of a probability distribution (Theorem \ref{BSM} and Corollary \ref{Uniqueness}). In particular, we show that Theorem \ref{general} is a corollary of Barnes multiplication in Theorem \ref{multiplic}. In Theorem \ref{scalinvar}
we establish the involution invariance of the components of the Selberg integral distribution using the scaling invariance of Barnes beta distributions in Theorem \ref{barnesbetascaling} and apply it to the Mellin transform in Corollary \ref{Mtransforminvol}.
\begin{theorem}\label{general}
Let $a\triangleq(a_1,\,a_2),$ $a_i>0,$ and $x\triangleq
(x_1,\,x_2),$ $x_1,\,x_2>0,$ then 
\begin{equation}\label{eqgeneral}
{\bf E}\big[M_{(a, x)}^q\bigr] \triangleq
\frac{\Gamma_2(x_1-q\,|\,a)}{\Gamma_2(x_1\,|\,a)}
\frac{\Gamma_2(x_2-q\,|\,a)}{\Gamma_2(x_2\,|\,a)}
\frac{\Gamma_2(a_1+a_2-q\,|\,a)}{\Gamma_2(a_1+a_2\,|\,a)}
\frac{\Gamma_2(x_1+x_2-q\,|\,a)}{\Gamma_2(x_1+x_2-2q\,|\,a)}
\end{equation}
is the Mellin transform of a probability distribution $M_{(a, x)}$
on $(0,\,\infty).$ Let $L$ be lognormal
\begin{equation}\label{Ldefin}
L \triangleq \exp\bigl(\mathcal{N}(0,\,4\log 2/a_1a_2)\bigr),
\end{equation}
and let $X_1,\,X_2,\,X_3$ have the $\beta^{-1}_{2, 2}(a, b)$
distribution with the parameters
\begin{align}
X_1 &\triangleq \beta_{2,2}^{-1}\Bigl(a,
b_0=x_1,\,b_1=b_2=(x_2-x_1)/2\Bigl), \label{X1}\\
X_2 & \triangleq \beta_{2,2}^{-1}\Bigl(a,
b_0=(x_1+x_2)/2,\,b_1=a_1/2,\,b_2=a_2/2\Bigr), \label{X2}\\
X_3 & \triangleq \beta_{2,2}^{-1}\Bigl(a, b_0=a_1+a_2,\,
b_1=b_2=(x_1+x_2-a_1-a_2)/2\Bigl). \label{X3}
\end{align}
Then, $M_{(a, x)}$ has the factorization
\begin{equation}\label{generaldecomp}
M_{(a, x)} \overset{{\rm in \,law}}{=}
2^{-\bigl(2(x_1+x_2)-(a_1+a_2)\bigr)/a_1a_2}\, L\,X_1\,X_2\,X_3.
\end{equation}
In particular, $\log M_{(a, x)}$ is absolutely continuous and
infinitely divisible.
\end{theorem}
\begin{proof}
Recalling the Mellin transform of $\beta_{2,2}(a, b)$ and
definition of $X_1, X_2, X_3$ in \eqref{X1}--\eqref{X3}, we
can write for ${\bf E}\bigl[X_1^q\bigr]{\bf E}\bigl[X_2^q\bigr]{\bf
E}\bigl[X_3^q\bigr]$ after some simplification
\begin{gather}
\frac{\Gamma_2(x_1-q\,|\,a)}{\Gamma_2(x_1\,|\,a)}
\frac{\Gamma_2(x_2-q\,|\,a)}{\Gamma_2(x_2\,|\,a)}
\frac{\Gamma_2(a_1+a_2-q\,|\,a)}{\Gamma_2(a_1+a_2\,|\,a)}
\frac{\Gamma_2(x_1+x_2-q\,|\,a)}{\Gamma_2(x_1+x_2\,|\,a)} \star
\nonumber \\
\star
\frac{\Gamma_2((x_1+x_2)/2\,|\,a)}{\Gamma_2((x_1+x_2)/2-q\,|\,a)}
\frac{\Gamma_2((x_1+x_2+a_1)/2\,|\,a)}{\Gamma_2((x_1+x_2+a_1)/2-q\,|\,a)}
\frac{\Gamma_2((x_1+x_2+a_2)/2\,|\,a)}{\Gamma_2((x_1+x_2+a_2)/2-q\,|\,a)}\star
\nonumber \\
\star
\frac{\Gamma_2((x_1+x_2+a_1+a_2)/2\,|\,a)}{\Gamma_2((x_1+x_2+a_1+a_2)/2-q\,|\,a)}
.
\end{gather}
Theorem \ref{multiplic} gives us
\begin{gather}
\Gamma_2(x_1+x_2-2q\,|\,a) = 2^{-B_{2, 2}(x_1+x_2-2q\,|\,a)/2}
\Gamma_2((x_1+x_2)/2-q\,|\,a)
\star\nonumber \\ \star
\Gamma_2((x_1+x_2+a_1)/2-q\,|\,a)\star\Gamma_2((x_1+x_2+a_2)/2-q\,|\,a)
\star\nonumber \\ \star
\Gamma_2((x_1+x_2+a_1+a_2)/2-q\,|\,a).
\end{gather}
Hence, we can write for ${\bf E}\bigl[X_1^q\bigr]{\bf
E}\bigl[X_2^q\bigr]{\bf E}\bigl[X_3^q\bigr]$
\begin{align}
{\bf E}\bigl[X_1^q\bigr]{\bf
E}\bigl[X_2^q\bigr]{\bf E}\bigl[X_3^q\bigr] & =
\frac{2^{B_{2, 2}(x_1+x_2\,|\,a)/2}}{2^{B_{2,
2}(x_1+x_2-2q\,|\,a)/2}}
\frac{\Gamma_2(x_1-q\,|\,a)}{\Gamma_2(x_1\,|\,a)}
\frac{\Gamma_2(x_2-q\,|\,a)}{\Gamma_2(x_2\,|\,a)}\star \nonumber \\ & \star
\frac{\Gamma_2(a_1+a_2-q\,|\,a)}{\Gamma_2(a_1+a_2\,|\,a)}
\frac{\Gamma_2(x_1+x_2-q\,|\,a)}{\Gamma_2(x_1+x_2-2q\,|\,a)}.
\end{align}
Now, using the identity
\begin{equation}\label{B22}
B_{2,2}(x\,|\,a) =
\frac{x^2}{a_1a_2}-\frac{x(a_1+a_2)}{a_1a_2}+\frac{a_1^2+3a_1a_2+a_2^2}{6a_1a_2},
\end{equation}
and the definition of $L$ in \eqref{Ldefin}, we can write
\begin{align}
{\bf E}\bigl[L^q\bigr]{\bf E}\bigl[X_1^q\bigr]{\bf
E}\bigl[X_2^q\bigr]{\bf E}\bigl[X_3^q\bigr]= & 2^{\bigl(2(x_1+x_2)-(a_1+a_2)\bigr)q/a_1a_2}\frac{\Gamma_2(x_1-q\,|\,a)}{\Gamma_2(x_1\,|\,a)}
\frac{\Gamma_2(x_2-q\,|\,a)}{\Gamma_2(x_2\,|\,a)}
\star\nonumber \\ & \star
\frac{\Gamma_2(a_1+a_2-q\,|\,a)}{\Gamma_2(a_1+a_2\,|\,a)}
\frac{\Gamma_2(x_1+x_2-q\,|\,a)}{\Gamma_2(x_1+x_2-2q\,|\,a)}.
\end{align}
This proves that the expression on the right-hand side of
\eqref{eqgeneral} is in fact the Mellin transform of the probability
distribution on $(0,\,\infty)$ that is given by the right-hand side of \eqref{generaldecomp}.
Its properties follow from the known properties of the normal and $\beta_{2,2}(a, b)$ distributions.
\end{proof}

We now proceed to give the new proof of our results on the Selberg integral that is based on Theorem \ref{general}. The starting point is the following remarkable formula due to Selberg
\cite{Selberg}. Given $\tau>1,$ $\lambda_i>-1/\tau,$ and $1\leq l<
\tau,$
\begin{gather}
\int\limits_{[0,\,1]^l} \prod_{i=1}^l s_i^{\lambda_1}(1-s_i)^{\lambda_2}\, \prod\limits_{i<j}^l |s_i-s_j|^{-2/\tau} ds_1\cdots ds_l =
\nonumber \\ \prod_{k=0}^{l-1}\frac{\Gamma(1-(k+1)/\tau)
\Gamma(1+\lambda_1-k/\tau)\Gamma(1+\lambda_2-k/\tau)}
{\Gamma(1-1/\tau)\Gamma(2+\lambda_1+\lambda_2-(l+k-1)/\tau)}, \label{Selberg}
\end{gather}
confer \cite{ForresterBook} for a modern treatment.
The reason for restricting $\tau>1$ is that in this case it is known that
there exists a positive probability distribution (constructed by integrating $s^{\lambda_1}(1-s)^{\lambda_2}$ with respect to the limit lognormal stochastic measure) such that its $l$th
moment is given by the left-hand side of \eqref{Selberg}, confer \cite{Me} and
references therein. In an attempt to compute this distribution in \cite{Me} we
constructed a probability distribution having this property.
\begin{theorem}\label{BSM}
Let $\tau>1$  and $\lambda_i>-1/\tau.$ Define the
Mellin transform
\begin{align}
{\bf E}\bigl[M^q_{(\tau, \lambda_1, \lambda_2)}\bigr] & \triangleq
\tau^{\frac{q}{\tau}} (2\pi)^{q}\,\Gamma^{-q}\bigl(1-1/\tau\bigr)
\frac{\Gamma_2(1-q+\tau(1+\lambda_1)\,|\,\tau)}{\Gamma_2(1+\tau(1+\lambda_1)\,|\,\tau)}\star
\nonumber \\ & \star
\frac{\Gamma_2(1-q+\tau(1+\lambda_2)\,|\,\tau)}{\Gamma_2(1+\tau(1+\lambda_2)\,|\,\tau)}
\frac{\Gamma_2(-q+\tau\,|\,\tau)}{\Gamma_2(\tau\,|\,\tau)}\star
\nonumber \\ & \star
\frac{\Gamma_2(2-q+\tau(2+\lambda_1+\lambda_2)\,|\,\tau)}{\Gamma_2(2-2q+\tau(2+\lambda_1+\lambda_2)\,|\,\tau)}
\label{M}
\end{align}
for $\Re(q)<\tau.$ Then, $M_{(\tau, \lambda_1, \lambda_2)}$ is a
probability distribution on $(0,\infty)$ and its positive ($0\leq l <
\tau$) and negative ($l\in\mathbb{N}$) integral moments satisfy
\begin{align}
{\bf E}\bigl[M^l_{(\tau, \lambda_1, \lambda_2)}\bigr] & =
\prod_{k=0}^{l-1} \frac{\Gamma(1-(k+1)/\tau)}{\Gamma(1-1/\tau)}
\frac{\Gamma(1+\lambda_1-k/\tau)\Gamma(1+\lambda_2-k/\tau)}
{\Gamma(2+\lambda_1+\lambda_2-(l+k-1)/\tau)}, \label{Selbergmomentsverified} \\
{\bf E}\bigl[M^{-l}_{(\tau, \lambda_1, \lambda_2)}\bigr] & =
\prod_{k=0}^{l-1}
\frac{\Gamma\bigl(2+\lambda_1+\lambda_2+(l+2+k)/\tau\bigr)
\Gamma\bigl(1-1/\tau\bigr) }{
\Gamma\bigl(1+\lambda_1+(k+1)/\tau\bigr)\Gamma\bigl(1+\lambda_2+(k+1)/\tau\bigr)
\Gamma\bigl(1+k/\tau\bigr) }. \label{Selbergmomentsnegative}
\end{align}
$\log M_{(\tau, \lambda_1, \lambda_2)}$ is absolutely continuous and
infinitely divisible.
\end{theorem}
\begin{remark}
In the special case of $\lambda_1=\lambda_2=0$ Theorem \ref{BSM}
first appeared in
\cite{Me4}. The general case was first considered by \cite{FLDR},
who gave an equivalent expression for the right-hand side of
\eqref{M} and so matched the positive integral moments without proving analytically that it
corresponds to a probability distribution. The first proof of Theorem \ref{BSM}
in full generality was given in \cite{Me}.
\end{remark}
\begin{corollary}\label{Uniqueness}
The Stieltjes moment problems for $M_{(\tau, \lambda_1, \lambda_2)}$ and $M_{(\tau, \lambda_1, \lambda_2)}^{-1}$ are indeterminate.
Let $\widetilde{M}_{(\tau, \lambda_1, \lambda_2)}$ be a probability distribution on $(0,\,\infty)$ such that
\begin{equation}
\widetilde{M}_{(\tau, \lambda_1, \lambda_2)} \overset{{\rm in \,law}}{=} L\,N,
\end{equation}
\begin{equation}\label{NYdef}
L \triangleq \exp\bigl(\mathcal{N}(0,\,4\log 2/\tau)\bigr),
\end{equation}
\emph{i.e.} $\log L$ is a zero-mean normal with variance $4\log
2/\tau,$
and $N$ is some distribution that is independent of $L.$ If the negative moments of
$\widetilde{M}_{(\tau, \lambda_1, \lambda_2)}$ equal those of $M_{(\tau, \lambda_1, \lambda_2)}$ in \eqref{Selbergmomentsnegative},
then $\widetilde{M}_{(\tau, \lambda_1, \lambda_2)}\overset{{\rm in \,law}}{=}M_{(\tau, \lambda_1, \lambda_2)}.$
\end{corollary}
\begin{proof}[Proof of Theorem \ref{BSM} and Corollary \ref{Uniqueness}]
Define
\begin{equation}\label{xidef}
x_i(\tau,\lambda_i) \triangleq 1+\tau(1+\lambda_i),\,i=1,2,\,a_1=1, \,a_2=\tau,
\end{equation}
and let $L(\tau)$ and $X_i(\tau,\lambda)$ be as in the statement of Theorem \ref{general}
corresponding to $a(\tau)=(1,\tau)$ and $x(\tau,\lambda)=\bigl(x_1(\tau,\lambda_1), x_2(\tau,\lambda_2)\bigr).$ In addition, define
\begin{equation}\label{Ydef}
Y(\tau) \triangleq \tau\,y^{-1-\tau}\exp\bigl(-y^{-\tau}\bigr)\,dy,\; y>0,
\end{equation}
\emph{i.e.} $Y(\tau)$ is a power of the exponential random variable (Fr\'echet distribution).
By the functional equation of $\Gamma_2$ in \eqref{feq}
and the definition of $Y(\tau)$ in \eqref{Ydef}, we have
\begin{align}
{\bf E}[Y(\tau)^q] & =\Gamma(1-q/\tau), \label{YMellin} \\
\Gamma_{2}(\tau-q\,|\,\tau) & =
\frac{\tau^{(\tau-q)/\tau-1/2}}{\sqrt{2\pi}}{\bf E}[Y(\tau)^q]\,\Gamma_2\bigl(1+\tau-q\,|\,\tau\bigr). \label{Ytransform}
\end{align}
Then, given Theorem \ref{general}, the difference between \eqref{eqgeneral} and \eqref{M} is in the third gamma ratio. Define
\begin{equation}\label{Decomposition}
M_{(\tau, \lambda_1, \lambda_2)} \triangleq 2\pi\,
2^{-\bigl[3(1+\tau)+2\tau(\lambda_1+\lambda_2)\bigr]/\tau}\,
\frac{
L(\tau)\,X_1(\tau,\lambda)\,X_2(\tau,\lambda)\,X_3(\tau,\lambda)\,Y(\tau)}{\Gamma\bigl(1-1/\tau\bigr)}.
\end{equation}
It is now elementary to see that the Mellin transform of $M_{(\tau, \lambda_1, \lambda_2)}$ equals the expression in \eqref{M}. The appearance of the $Y$ distribution in \eqref{Decomposition} is to account for this difference in the third gamma ratio.
The proof of \eqref{Selbergmomentsverified} and \eqref{Selbergmomentsnegative} is now immediate from \eqref{feq}, which implies the identity
\begin{equation}
\frac{\Gamma_2(x-l\,|\,\tau)}{\Gamma_2(x\,|\,\tau)} =
\prod\limits_{i=0}^{l-1}
\Gamma_1\bigl(x-(i+1)\,|\,\tau\bigr),\;l\in\mathbb{N},
\end{equation}
where $\Gamma_1(x\,|\,\tau)$ is defined in \eqref{gamma1}. The
remaining computation, which determines the overall constant in \eqref{Decomposition}, is straightforward. 
The proof of Corollary \ref{Uniqueness} follows from \eqref{Decomposition} due to
the determinacy of the Stieltjes moment problems for $\beta_{2,2}$ (compactly supported) and $Y^{-1}$ (Carleman's criterion)
and its indeterminacy for $L, \,L^{-1}$ (lognormal) and $\beta^{-1}_{2,2}, \,Y$ (infinite moments), confer \cite{Char}, Sections 2.2 and 2.3.
\end{proof}

\begin{remark} The Mellin transform of $M_{(\tau, \lambda_1, \lambda_2)}$
has two additional properties that we established in \cite{Me}. We record them here for completeness as they are important for Theorem \ref{Deriv} below. 
It satisfies the infinite product representation
\begin{align}
{\bf E}\bigl[M^q_{(\tau, \lambda_1, \lambda_2)}\bigr] & = 
\frac{\tau^{q}\Gamma\bigl(1-q/\tau\bigr)\Gamma\bigl(2-2q+\tau(1+\lambda_1+\lambda_2)\bigr)}{
\Gamma^{q}\bigl(1-1/\tau\bigr)\Gamma\bigl(2-q+\tau(1+\lambda_1+\lambda_2)\bigr)}\star \nonumber \\ & \star
\prod\limits_{m=1}^\infty \left(m\tau\right)^{2q}
\frac{\Gamma\bigl(1-q+m\tau\bigr)}{\Gamma\bigl(1+m\tau\bigr)}
\frac{\Gamma\bigl(1-q+\tau\lambda_1+m\tau\bigr)}{\Gamma\bigl(1+\tau\lambda_1+m\tau\bigr)}\star \nonumber \\
& \star
\frac{\Gamma\bigl(1-q+\tau\lambda_2+m\tau\bigr)}{\Gamma\bigl(1+\tau\lambda_2+m\tau\bigr)} 
\frac{\Gamma\bigl(2-q+\tau(\lambda_1+\lambda_2)+m\tau\bigr)}{\Gamma\bigl(2-2q+\tau(\lambda_1+\lambda_2)+m\tau\bigr)}
\label{Minfinprod}
\end{align}
and the functional equations
\begin{gather}
{\bf E}\bigl[M^q_{(\tau, \lambda_1, \lambda_2)}\bigr] =
{\bf E}\bigl[M^{q-\tau}_{(\tau, \lambda_1, \lambda_2)}\bigr]\tau
(2\pi)^{\tau-1}\Gamma^{-\tau}\Bigl(1-\frac{1}{\tau}\Bigr)
\Gamma(\tau-q) \star \nonumber \\ \star
\frac{\Gamma\bigl((1+\lambda_1)\tau-(q-1)\bigr)\Gamma\bigl((1+\lambda_2)\tau-(q-1)\bigr)}{\Gamma\bigl((2+\lambda_1+\lambda_2)\tau-(2q-2)\bigr)}
\frac{\Gamma\bigl((2+\lambda_1+\lambda_2)\tau-(q-2)\bigr)}{\Gamma\bigl((3+\lambda_1+\lambda_2)\tau-(2q-2)\bigr)}, \\
{\bf E}\bigl[M^q_{(\tau, \lambda_1, \lambda_2)}\bigr] =
{\bf E}\bigl[M^{q-1}_{(\tau, \lambda_1, \lambda_2)}\bigr] \,
\frac{\Gamma(1-q/\tau)\Gamma\bigl(2+\lambda_1+\lambda_2-(q-2)/\tau\bigr)}{\Gamma(1-1/\tau)}
\star \nonumber \\ \star
\frac{\Gamma\bigl(1+\lambda_1-(q-1)/\tau\bigr)\Gamma\bigl(1+\lambda_2-(q-1)/\tau\bigr)}
{\Gamma\bigl(2+\lambda_1+\lambda_2-(2q-2)/\tau\bigr)\Gamma\bigl(2+\lambda_1+\lambda_2-(2q-3)/\tau\bigr)}.
\end{gather}
\end{remark}

The last result that we will consider in this section has to do with the scaling invariance
in the context of the analytic extension of the Selberg integral in Theorem \ref{BSM}. Specifically, we are interested in the behavior of the decomposition in \eqref{Decomposition} under the involution $\tau\rightarrow 1/\tau.$ Let $a$ and $x$ be as in \eqref{xidef}.
\begin{theorem}\label{scalinvar} 
The distributions $L(\tau),$ $X_i(\tau,\lambda),$ and $M_{(a, x)}$ are involution invariant under $\tau\rightarrow 1/\tau$ and
$\lambda\rightarrow \tau\lambda.$
\begin{align}
L^{1/\tau}\bigl(\frac{1}{\tau}\bigr) &\overset{{\rm in \,law}}{=} L(\tau), \label{Linvar} \\
X^{1/\tau}_i\bigl(\frac{1}{\tau}, \tau\lambda\bigr) &\overset{{\rm in \,law}}{=}
X_i(\tau, \lambda), \;i=1,2,3, \label{Xinvar} \\
M^{1/\tau}_{\bigl(a(1/\tau), \,x(1/\tau,\tau\lambda)\bigr)} &\overset{{\rm in \,law}}{=} M_{\bigl(a(\tau), \,x(\tau,\lambda)\bigr)}. \label{Minvar}
\end{align}
\end{theorem}
\begin{proof}
We have by \eqref{Ldefin}
\begin{align}
{\bf E}\Bigl[L^{q/\tau} \bigl(\frac{1}{\tau}\bigr)\Bigr] & = {\bf E}\Bigl[e^{(q/\tau)\mathcal{N}(0,\,4\tau\log 2)}\Bigr], \nonumber \\
& = e^{4q^2\log 2/2\tau} \equiv {\bf E}\bigl[L^{q} (\tau)\bigr].
\end{align}
To prove \eqref{Xinvar}, observe the identity
\begin{equation}\label{xinvar}
x_i(\tau, \lambda_i)/\tau = x_i(1/\tau, \tau\lambda_i), \;i=1,2.
\end{equation}
Hence, by the definition of $X_i(\tau,\lambda)$ and
Theorem \ref{barnesbetascaling}, we have the identity
\begin{align}
X^{1/\tau}(1/\tau, \tau\lambda) & \overset{{\rm in \,law}}{=} \beta_{2,2}^{-1/\tau}\bigl(a/\tau, b/\tau\bigr), \nonumber \\
& \overset{{\rm in \,law}}{=} \beta_{2,2}^{-1}\bigl(a, b\bigr),
\end{align}
where $a=(1,\tau)$ and $b=(b_0, b_1, b_2)$ is given in terms of
$a$ and $x_i, \,i=1, 2$ in \eqref{X1}--\eqref{X3}.
The proof of \eqref{Minvar} follows from \eqref{Linvar} and \eqref{Xinvar}
(or can be seen directly from \eqref{eqgeneral} by Theorem \ref{scaling} and \eqref{B22}).
\end{proof}

Define the reduced distribution, confer \eqref{Decomposition}, that is defined for all $\tau>0.$
\begin{equation}
M^{\star}_{(\tau, \lambda_1, \lambda_2)} \triangleq
\frac{\Gamma(1-1/\tau)}{2\pi}\,M_{(\tau, \lambda_1, \lambda_2)}.
\end{equation}
\begin{corollary}\label{Mtransforminvol}
$M^{\star}_{(\tau,\lambda_1,\lambda_2)}$ is not involution invariant and satisfies under $\tau\rightarrow 1/\tau,$ $\lambda\rightarrow \tau\lambda$
\begin{equation}\label{Mtransform}
{\bf E}\bigl[(M^{\star})^{q/\tau}_{(1/\tau, \tau\lambda_1, \tau\lambda_2)}\bigr] =
\frac{\Gamma(1-q)}{\Gamma(1-q/\tau)}\,{\bf E}\bigl[(M^{\star})^q_{(\tau, \lambda_1, \lambda_2)}\bigr].
\end{equation}
\end{corollary}
\begin{proof}
$M^{\star}_{(\tau,\lambda_1,\lambda_2)}$ is not involution invariant because  $Y$ is not involution invariant, confer \eqref{YMellin}. Instead,
we have by \eqref{YMellin} the identity
\begin{equation}
{\bf E}\bigl[Y^{q/\tau}(1/\tau)\bigr] = \frac{\Gamma(1-q)}{\Gamma(1-q/\tau)}
\,{\bf E}\bigl[Y^q(\tau)\bigr].  
\end{equation}
The proof now follows from \eqref{Minvar}.
\end{proof}
\begin{remark}
The scaling behavior of the Mellin transform 
in Theorem \ref{BSM} was first noted in \cite{FLDR}
using a different technique and led to an interesting conjecture
on a freezing transition. Our probabilistic approach to this problem is new.
\end{remark}
\begin{remark}
It is interesting to point out that the $L$ and $X_i,$ $i=1,2,3$ distributions
on the one hand and $Y$ on the other in the structure of $M_{(\tau,\lambda_1,\lambda_2)}$ are intrinsically different. This can be seen on three levels. First, the proof of Theorem \ref{BSM} indicates that $LX_1X_2X_3$
appears as one block from Theorem \ref{general}, while $Y$ is only needed to match the moments given by Selberg's formula. Second, Theorem \ref{scalinvar} and Corollary \ref{Mtransforminvol} show that $L$ and $X_i,$ are involution invariant, whereas $Y$ is not. Finally, the law of the total mass of the limit
lognormal measure on the circle was conjectured in \cite{FyoBou} to be precisely
the same as $Y,$ while our conjecture for the law of the total mass of this measure on the unit interval is $LX_1X_2X_3$ times $Y$ so that it appears
that $Y$ comes from the circle and $LX_1X_2X_3$ is a superstructure that is needed to transform the law of the total mass from the circle to the unit interval.
\end{remark}

\section{Law of the Derivative Martingale, Barnes $G$ function, and Riemann Xi}
In this section we will consider a limit of the probability distribution
$M_{(\tau, \lambda_1, \lambda_2)}$ in Theorem \ref{BSM} as $\tau\downarrow 1.$
We conjectured in \cite{Me} that $M_{(\tau, \lambda_1, \lambda_2)}$ is
the law of the (generalized) total mass of the limit lognormal stochastic measure on the unit interval. Under the same conjecture, the limit that we will consider here is the law of the derivative martingale and is of a particular interest in the theory of critical limit lognormal chaos, confer \cite{barraletal} and \cite{dupluntieratal}. We will first show what our general results give in this case and then introduce the connection with the xi function. Throughout this section we let $\lambda_1=\lambda_2=0$ for simplicity.

We begin by briefly reminding the reader about the Barnes $G(z)$ function, confer \cite{Genesis} and \cite{SriCho} for review. It is related to the double gamma function by
\begin{equation}
G(z) = (2\pi)^{(z-1)/2}\,\frac{\Gamma^{-1}_2(z\,|\,1)}{\Gamma^{-1}_2(1\,|\,1)}
\end{equation}
and satisfies
\begin{align}
G(1) & =1, \\
G(z+1) & = \Gamma(z)\,G(z).
\end{align}
It is known that $G(z)$ is entire with zeroes at $z=-n, \,n\in\mathbb{N},$ of
multiplicity $n+1.$

Define the distribution
$M_c$ by taking the weak limit of $M_{(\tau, \lambda_1, \lambda_2)}.$
\begin{equation}
M_c \triangleq \lim\limits_{\tau\downarrow 1} \frac{M_{(\tau, \lambda_1=0, \lambda_2=0)}}{1-1/\tau}.
\end{equation}
\begin{theorem}\label{Deriv} The Mellin transform of $M_c$ is
\begin{equation}\label{Mderiv}
{\bf E}\bigl[M^q_c\bigr]  =
\frac{G(4-2q)}{G(1-q)\, G^2(2-q)\,G(4-q)}
\end{equation}
for $\Re(q)<1.$ It has zeroes at $q=n+5/2,$ $n\in\mathbb{N},$ of multiplicity $2n+2$ and
poles: at $q=1$ of order $1,$ at $q=2$ of order $3,$ and at $q=n,$ $n=3,4,5\cdots$ of order $2n-2.$ The Mellin transform satisfies the functional equation
\begin{equation}
{\bf E}\bigl[M^q_c\bigr] =
\frac{\Gamma(1-q)\,\Gamma^2(2-q)\,\Gamma(4-q)}{\Gamma(4-2q)\,\Gamma(5-2q)}
\,{\bf E}\bigl[M^{q-1}_c\bigr] 
\end{equation}
and an infinite product representation
\begin{equation}
{\bf E}\bigl[M^q_c\bigr] =
\frac{\Gamma\bigl(1-q\bigr)\Gamma\bigl(3-2q\bigr)}{
\Gamma\bigl(3-q\bigr)}
\prod\limits_{m=1}^\infty m^{2q}
\frac{\Gamma^3\bigl(1-q+m\bigr)}{\Gamma^3\bigl(1+m\bigr)}
\frac{\Gamma\bigl(2-q+m\bigr)}{\Gamma\bigl(2-2q+m\bigr)}.
\end{equation}
The negative moments of $M_c$ satisfy for $l\in\mathbb{N}$
\begin{equation}
{\bf E}\bigl[M^{-l}_c\bigr] =
\prod_{k=0}^{l-1}
\frac{(3+l+k)!}{(k+1)!^2 \,k!}.
\end{equation}
$M_c$ has the factorization
\begin{equation}\label{McDecomposition}
M_c = \frac{\pi}{32}\,L\,\,X_2\,X_3\,Y,
\end{equation}
where
\begin{align}
L & = \exp\bigl(\mathcal{N}(0, 4\log 2)\bigr) \;(\text{Lognormal}), \\
X_2 & = \beta^{-1}_{2,2}\bigl(a=(1,1), \,b_0=2, \,b_1=b_2=1/2\bigr), \\
X_3 & = \frac{2}{y^3}\,dy, \;y>1 \;(\text{Pareto}),\\
Y & = \frac{1}{y^2}\,e^{-1/y}\,dy,\;y>0 \;(\text{Fr\'echet}).
\end{align}
In particular, $\log M_c$ is infinitely divisible and absolutely continuous.
\end{theorem}
The proof follows directly from our results in Sections 3 -- 5.

It is clear from Theorem \ref{Deriv} that the most nontrivial component of
$M_c$ is $X_2.$ Consider more generally a family of $\beta_{2,2}$ distributions
that is parameterized by $\delta>0.$
\begin{equation}
\beta_{2,2}(\delta) \triangleq \beta_{2,2}\bigl(a=(1,1), \,b_0=\delta, \,b_1=b_2=1/2\bigr).
\end{equation}
Its Mellin transform is given in \eqref{beta22} and satisfies in terms of $G(z)$
\begin{align}
{\bf E}\bigl[\beta_{2,2}^q(\delta)\bigr] & = \frac{G(\delta)}{G(q+\delta)}
\frac{G^2(q+\delta+1/2)}{G^2(\delta+1/2)}
\frac{G(\delta+1)}{G(q+\delta+1)}, \\
& = \prod\limits_{k=0}^\infty \Bigl[
\frac{\delta+k}{q+\delta+k} \frac{(q+\delta+1/2+k)^2}
{(\delta+1/2+k)^2} \frac{\delta+1+k}
{q+\delta+1+k}\Bigr]^{k+1}
\end{align}
by Corollary \ref{BarnesFactorSpecial}.
Remarkably, this distribution is intrinsically related to the Riemann xi function. Using \eqref{LKH}, the Mellin transform of $\beta_{2,2}(\delta)$ is given by
\begin{align}
{\bf E}\bigl[\beta^q_{2, 2}(\delta)\bigr] & =
\exp\Bigl(\int\limits_0^\infty (e^{-tq}-1) e^{-\delta
t} \frac{(1-e^{-t/2})^2}{(1-e^{-t})^2} \frac{dt}{t}\Bigr), \nonumber \\
& = \exp\Bigl(\frac{1}{4} \int\limits_0^\infty (e^{-tq}-1) e^{-(\delta-1/2)
t} \frac{1}{\cosh^2(t/4)}\frac{dt}{t}\Bigr).
\end{align}
Now, the reader who is familiar with \cite{BiaPitYor} and \cite{PitYor}
will recognize the significance of the hyperbolic cosine in the L\'evy density
of  $-\log\beta_{2,2}(\delta).$ Recall that the Laplace transform of the $C_2$ distribution satisfies
\begin{equation}\label{C2Laplace}
{\bf E}\bigl[e^{-qC_2}\bigr] = \Bigl[\frac{1}{\cosh\sqrt{2q}}\Bigr]^2,\; q>0,
\end{equation}
confer \cite{BiaPitYor} for details. Hence, we can write
\begin{equation}\label{beta22mellinC2}
{\bf E}\bigl[\beta^q_{2, 2}(\delta)\bigr] = \exp\Bigl(\frac{1}{4} \int\limits_0^\infty (e^{-tq}-1) e^{-(\delta-1/2)
t} \,{\bf E}\bigl[e^{-t^2\,C_2/32}\bigr] \frac{dt}{t}\Bigr).
\end{equation}
On the other hand, the Mellin transform of the $C_2$ distribution (and of its kin $S_2,$ confer Section 7) is essentially the Riemann xi function. 
\begin{equation}\label{C2Mellin}
{\bf E}\bigl[C_2^q\bigr] = \frac{(2^{2q+2}-1)}{q+1}\bigl(\frac{2}{\pi}\bigr)^{q+1}\xi(2q+2),\;q\in\mathbb{C},
\end{equation}
where xi is defined in terms of the Riemann zeta function\footnote{Contrary to the commonly accepted usage, we use
$q$ as opposed to $s$ as the generic complex variable, to be consistent with the rest of the paper.} by
\begin{equation}\label{riemannxidef}
\xi(q) \triangleq \frac{1}{2}q(q-1)\pi^{-q/2}\Gamma(q/2)\zeta(q).
\end{equation}
Hence, the intrinsic relationship between $\beta_{2,2}(\delta)$
and the xi function. We will make this relationship more explicit for $0<\delta<1$ by computing the cumulants of  $-\log\beta_{2,2}(\delta),$
which we denote by $\kappa_n(\delta)$ so that
\begin{equation}
\kappa_n(\delta) \triangleq (-1)^n\frac{d^n}{dq^n}\vert_{q=0} \log {\bf E}\bigl[\beta_{2,2}^q(\delta)\bigr].
\end{equation}
We now have the following result.
\begin{theorem}\label{Gxi}
For $0<\delta<1,$ and $n=1,2,3\cdots,$ the $n$th cumulant is
\begin{equation}\label{cumrestricted}
\kappa_n(\delta)  = \frac{1}{8} \sum\limits_{m=0}^\infty \frac{\bigl(-(\delta-1/2)\bigr)^m}{m!} \,32^{(m+n)/2} \,\Gamma\bigl(\frac{m+n}{2}\bigr) \,{\bf E}\bigl[C_2^{-(m+n)/2}\bigr].
\end{equation}
In particular, for $\delta=1/2,$ we obtain
\begin{equation}\label{cumxi}
\kappa_n(1/2) = \frac{32^{n/2}}{8} \Gamma\bigl(\frac{n}{2}\bigr)
\frac{(2^{2-n}-1)}{1-n/2}\bigl(\frac{2}{\pi}\bigr)^{1-n/2}\,\xi(n-1).
\end{equation}
\end{theorem}
\begin{proof}
This follows from \eqref{beta22mellinC2} and \eqref{C2Mellin} by changing variables $y=t^2\,C_2/32,$ expanding the exponential in power series,
and using $\xi(2-n)=\xi(n-1)$ to obtain \eqref{cumxi}. The resulting series is only convergent for $0<\delta<1.$
\end{proof}
\begin{remark}
The relationship between $\beta_{2,2}(\delta)$ and $C_2$ is not fully understood. Like $-\log\beta_{2,2}(\delta),$ $C_2$ is an infinitely divisible and absolutely continuous probability distribution on $(0, \,\infty).$ It is an open question how to relate them beyond \eqref{beta22mellinC2}. In particular,
what, if anything, can be inferred about $C_2$ and its Mellin transform from the structure of $\beta_{2,2}(\delta)$ or $M_c?$
\end{remark}

\section{An Approximation of the Riemann Xi Function}

We remind the reader of the definition of the $S_2$ distribution and its relationship to the Riemann xi function
following \cite{BiaPitYor}, which is the primary reference as well
as motivation for the results of this section.
The Riemann xi has a remarkable probabilistic representation,
originally due to Riemann, which in the modern
language can be written in the form
\begin{equation}\label{S2xi}
\Bigl(\frac{2}{\pi}\Bigr)^{q}2\xi(2q) = {\bf E}\bigl[S_2^q\bigr],\;
q\in\mathbb{C},
\end{equation}
where $S_2$ is a probability distribution on $(0, \,\infty)$ that is
defined by
\begin{equation}
S_2 \triangleq \frac{2}{\pi^2} \sum\limits_{n=1}^\infty
\frac{\Gamma_{2,n}}{n^2},
\end{equation}
and $\{\Gamma_{2,n}\}$ denotes an iid family of gamma distributions on
$(0, \,\infty)$ with the density $x e^{-x}.$ Recalling the Jacobi
theta function $\theta(t)$ (strictly speaking, a special case of
Jacobi's $\theta_3$)
\begin{equation}
\theta(t)\triangleq 1+2\sum\limits_{n=1}^\infty e^{-\pi tn^2},\;t>0,
\end{equation}
the Laplace transform of $S_2$ can be written as\footnote{We mention in passing
that $S_2$ as defined by \eqref{S2Laplace1} appears also 
in a model of Anderson localization in the context of statistics of eigenvectors of random banded matrices, confer \cite{YM}.}
\begin{align}
{\bf E}\bigl[e^{-qS_2}\bigr] & = \Bigl[\frac{\sqrt{2q}}{\sinh\sqrt{2q}}\Bigr]^2, \label{S2Laplace1}\\
& = \exp\Bigl(\int\limits_{0}^\infty
(e^{-qt}-1)\bigl(\theta(\frac{\pi
t}{2})-1\bigr)\frac{dt}{t}\Bigr), \label{S2Laplace2}\; q>0.
\end{align}
In particular, $S_2$ is infinitely divisible and absolutely
continuous.\footnote{We note that $\theta(t)\propto t^{-1/2}$ as
$t\rightarrow 0$ and $\theta(t)-1\propto e^{-\pi t}$ as $t\rightarrow
+\infty$ so that $\bigl(\theta(\pi t/2)-1\bigr)/t$ is a valid L\'evy
density.}

This construction inspired us to try to find a probabilistic
representation of $\xi$ in terms of the Mellin transform of the logarithm of Barnes beta distributions. The relevance of the latter is that the Jacobi
triple product for $\theta(t)$ naturally leads to a limit of Barnes
beta distributions.
\begin{equation}
\theta(t) = \prod\limits_{n=1}^\infty (1-e^{-2n\pi
t})\prod\limits_{n=1}^\infty (1+e^{-(2n-1)\pi
t})\prod\limits_{n=1}^\infty (1+e^{-(2n-1)\pi t}),
\end{equation}
confer \cite{WhitWat}, Chapter 21. Define a sequence of approximations
\begin{align}
\theta_M(t)  \triangleq \prod\limits_{n=1}^M (1-e^{-2n\pi
t})\Bigl[\prod\limits_{n=1}^M (1+e^{-(2n-1)\pi
t})\Bigr]^2 = & \prod\limits_{n=1}^M (1-e^{-2n\pi
t}) \star \nonumber \\ & \star  \Bigl[\prod\limits_{n=1}^M \frac{(1-e^{-(4n-2)\pi
t})}{(1-e^{-(2n-1)\pi
t})}\Bigr]^2.
\end{align}
The corresponding Barnes beta sequence $\beta_M(\delta)$ is
a family of distributions on $(0, 1]$ that is parameterized
by $M\in\mathbb{N}$ and $\delta>0.$ Let $i, j=1\cdots M$ and $k=1,2.$ By Theorem \ref{main},
\begin{gather}
\beta_M(\delta) \triangleq \beta_{2M,3M}\Bigl(a^{(k)}_i=(2i-1)\frac{\pi^2}{2},\,b_0=\delta,\,b^{(k)}_j=(4j-2)\frac{\pi^2}{2},\,
b^{(3)}_j=2j\frac{\pi^2}{2}\Bigr), \label{betaMdelta} \\
{\bf E}\bigl[e^{q\log\beta_M(\delta)}\bigr] = \exp\Bigl(\int\limits_{0}^\infty
(e^{-qt}-1)e^{-\delta t}\theta_M(\frac{\pi
t}{2})\frac{dt}{t}\Bigr),\; \Re(q)>-\delta.
\end{gather}
Define the limit distribution $T(\delta)$ on $(0,\,\infty)$ by\footnote{\({\bf P}[\beta_M(\delta)=1]\rightarrow 0\) as \(M\rightarrow \infty\) by \eqref{Pof1}.}
\begin{equation}
\beta_M(\delta) \overset{{\rm in \,law}}{\rightarrow} \exp\bigl(-T(\delta)\bigr),\,M\rightarrow \infty.
\end{equation}
By construction, as $\theta_M(t)\rightarrow\theta(t)$ in the limit $M\rightarrow \infty,$ we have
\begin{equation}\label{TLaplace}
{\bf E}\bigl[e^{-qT(\delta)}\bigr]  =
\exp\Bigl(\int\limits_{0}^\infty (e^{-qt}-1)e^{-\delta t}\theta(\frac{\pi
t}{2})\frac{dt}{t}\Bigr),\; q>0, \,\delta>0,
\end{equation}
so that $T(\delta)$ is absolutely continuous and infinitely divisible.
In light of \eqref{S2Laplace2}, one naturally wants to know how $T(\delta)$ is related to
$S_2$ and how its Mellin transform is related to $\xi.$ These questions are answered in the following
theorem. Recall the definition of $\beta_{0,0},$ in particular, $\beta_{0, 0}$
has density $b_0\,x^{b_0-1}\,dx$ on \((0,1).\)
\begin{theorem}\label{Txi}
Let $\delta>0$ and $\beta_{00}(b_0=\delta)$ be independent of the other distributions.
\begin{equation}\label{Txieq1}
T(\delta) \overset{{\rm in \,law}}{=} \sum\limits_{n=1}^\infty
\frac{\Gamma_{2,n}}{\pi^2 n^2/2+\delta} + \bigl(-\log\beta_{00}(\delta)\bigr).
\end{equation}
Denote the density of $T(\delta)$ by $f_{T(\delta)}(x).$
\begin{equation}\label{Txieq2}
f_{T(\delta)}(x) = \delta e^{-\delta x}
\Bigl[\frac{\sinh\sqrt{2\delta}} {\sqrt{2\delta}}\Bigr]^2\,{\bf
P}(S_2<x),\;x>0.
\end{equation}
The Mellin transform of $T(\delta)$ is entire
and satisfies the functional equation for $\delta<\pi^2/2$
\begin{equation}\label{Txieq3}
{\bf E}[T(\delta)^q] -\frac{q}{\delta}{\bf E}[T(\delta)^{q-1}]  =
\Bigl[\frac{\sinh\sqrt{2\delta}} {\sqrt{2\delta}}\Bigr]^2\,
\Bigl(\frac{2}{\pi}\Bigr)^q \sum\limits_{n=0}^\infty
\frac{1}{n!}\Bigl(\frac{-2\delta}{\pi}\Bigr)^n\,2\xi(2q+2n).
\end{equation}
In particular, by taking the limit $\delta\rightarrow 0,$ we obtain
\begin{equation}
\lim\limits_{\delta\rightarrow 0} \Bigl[{\bf E}[T(\delta)^q] -\frac{q}{\delta}{\bf E}[T(\delta)^{q-1}]\Bigr] =
\Bigl(\frac{2}{\pi}\Bigr)^q
\,2\xi(2q).
\end{equation}

\end{theorem}
\begin{remark}\label{Splitting} By Theorem \ref{Reduction}, $\beta_M(\delta)$ splits into a product of $\beta_{0,M}\bigl(b_0, b_1^{(3)}\cdots b_M^{(3)}\bigr)$ so that
$T(\delta)$ is the $M\rightarrow\infty$ limit of a sum of 
$-\log\beta_{0,M}\bigl(b_0, b_1^{(3)}\cdots b_M^{(3)}\bigr)$ with different values of $b_0,$ confer \eqref{betaMdelta}. 
The details of this calculation and its implications are left to further research. 
\end{remark}
\begin{remark} The Jacobi triple product plays the central role in the relationship between Barnes beta distributions and $T(\delta).$ It is interesting to point out that the Jacobi triple product is also of
central importance in the proof of the key probabilistic property of
the two-variable zeta function in \cite{LagRains}. We are not aware
of any connection between the two results and leave it as an intriguing open question.
\end{remark}

The proof of Theorem \ref{Txi} requires an auxiliary result that we state and prove first.
We need to define a one-parameter family of infinitely divisible
distributions $S_2(\delta)$ on $(0,\,\infty)$ extending $S_2=S_2(0).$ Let
$f_{S_2}(x)$ denote the probability density of $S_2.$
\begin{definition} Given $\delta>0,$
\begin{equation}
S_2(\delta) \triangleq \sum\limits_{n=1}^\infty
\frac{\Gamma_{2,n}}{\pi^2 n^2/2+\delta}.
\end{equation}
\end{definition}
\begin{lemma}\label{S2aux}
$S_2(\delta)$ is infinitely divisible and absolutely continuous.
Denote its density by $f_{S_2(\delta)}(x).$ Then, its Laplace
transform, density, and Mellin transform satisfy
\begin{align}
{\bf E}\bigl[e^{-qS_2(\delta)}\bigr] & =
\Bigl[\frac{\sinh\sqrt{2\delta}}{\sqrt{2\delta}}\Bigr]^2\, {\bf E}\bigl[e^{-(q+\delta)S_2}\bigr], \label{Seq1}\\
& = \exp\Bigl(\int\limits_{0}^\infty
(e^{-qt}-1)e^{-\delta
t}\bigl(\theta(\frac{\pi t}{2})-1\bigr)\frac{dt}{t}\Bigr),\; q>0, \,\delta>0, \label{Seq2}\\
f_{S_2(\delta)}(x) & =
\Bigl[\frac{\sinh\sqrt{2\delta}}{\sqrt{2\delta}}\Bigr]^2\,e^{-\delta
x}f_{S_2}(x), \; x>0, \,\delta>0, \label{Seq3} \\
{\bf E}[S_2(\delta)^q] & = \Bigl[\frac{\sinh\sqrt{2\delta}}
{\sqrt{2\delta}}\Bigr]^2\, \Bigl(\frac{2}{\pi}\Bigr)^q
\sum\limits_{n=0}^\infty
\frac{1}{n!}\Bigl(\frac{-2\delta}{\pi}\Bigr)^n\,2\xi(2q+2n),\;q\in\mathbb{C},\,
\delta<\pi^2/2. \label{Seq4}
\end{align}
\end{lemma}
\begin{proof}
The starting point is the formula given in \cite{BiaPitYor}, Section 3.2, for the L\'evy density
$\rho(t)$ of the weighted sum of positive, independent, infinitely divisible distributions of the
form $\sum_n c_n\, X_n,$
where $c_n>0$ and $X_n$ has L\'evy density $\rho(t)$ for all $n.$
\begin{equation}
\rho(t) = \sum_n \frac{1}{c_n} \rho(t/c_n).
\end{equation}
The L\'evy density of $\Gamma_{2,n}$ is $2e^{-t}/t$ so that the L\'evy density $\rho_{S_2(\delta)}(t)$ of $S_2(\delta)$ is
\begin{equation}
\rho_{S_2(\delta)}(t) = \frac{e^{-\delta t}}{t} \bigl(\theta(\pi t/2)-1\bigr).
\end{equation}
Then, the Laplace transform of $S_2(\delta)$ can be written as
\begin{align}
{\bf E}\bigl[e^{-qS_2(\delta)}\bigr] & =
\exp\Bigl(\int\limits_{0}^\infty
(e^{-qt}-1)\rho_{S_2(\delta)}(t)\,dt\Bigr), \label{Saux1}\\
& = \Bigl[\prod\limits_{n=1}^\infty \frac{\delta+\pi^2n^2/2}{q+\delta+\pi^2n^2/2}\Bigr]^2, \label{Saux2} \\
& = \Bigl[\frac{\sinh\sqrt{2\delta}}{\sqrt{2\delta}}\Bigr]^2\, {\bf E}\bigl[e^{-(q+\delta)S_2}\bigr], \label{Saux3}
\end{align}
where we used Frullani's formula for $\log(x)$  
to evaluate the integral in \eqref{Saux1} and the infinite product representation of $\sinh(x)$ and \eqref{S2Laplace1} to obtain \eqref{Saux3}. This proves \eqref{Seq1} and \eqref{Seq2}. The density of $S_2(\delta)$ follows from \eqref{Seq1} so that the Mellin transform 
is
\begin{equation}
{\bf E}[S_2(\delta)^q] = \Bigl[\frac{\sinh\sqrt{2\delta}}
{\sqrt{2\delta}}\Bigr]^2\, \int\limits_0^\infty x^q e^{-\delta x}f_{S_2}(x)\,dx.
\end{equation}
Expanding the exponential and making use of \eqref{S2xi}, we obtain \eqref{Seq4}, provided
that the integral can be computed term by term. The partial sums of $\exp(-\delta x)$ are
bounded by $\exp(\delta x).$ If $\delta<\pi^2/2,$ then $\exp(\delta x) f_{S_2}(x)$ is exponentially
small as $x\rightarrow \infty,$ confer Table 1 in \cite{BiaPitYor} so that the result follows by dominated
convergence. The series is absolutely convergent if $\delta<\pi^2/2$ as is clear from
\eqref{riemannxidef} since $\zeta(q)\rightarrow 1$ as $\Re(q)\rightarrow +\infty.$
\end{proof}
\begin{proof}[Proof of Theorem \ref{Txi}]
The proof of \eqref{Txieq1} follows from \eqref{TLaplace}, \eqref{Seq2} and
\begin{equation}
{\bf E}\Bigl[\exp\bigl(q\log\beta_{00}(\delta)\bigr)\Bigr] = \exp\Bigl(\int\limits_0^\infty(e^{-qt}-1)e^{-\delta t} \frac{dt}{t}\Bigr).
\end{equation}
The density of $T(\delta)$ in \eqref{Txieq2} is the convolution of the density of $S_2(\delta)$ in \eqref{Seq3} and the density $\delta\exp(-\delta x)$ of
$-\log\beta_{00}(\delta).$  Since the cumulative distribution function ${\bf P}(S_2<x)$ of $S_2$
is exponentially small as $x\rightarrow 0,$ confer Proposition 2.1 in \cite{BiaPitYor}, the Mellin transform of $T(\delta)$ is entire
by \eqref{Txieq2}. To prove \eqref{Txieq3}, we integrate by parts and use \eqref{Seq4}.
\begin{align}
{\bf E}[T(\delta)^q] & = -\Bigl[\frac{\sinh\sqrt{2\delta}} {\sqrt{2\delta}}\Bigr]^2
\int\limits_0^\infty x^q {\bf P}(S_2<x) \frac{d}{dx} e^{-\delta x} \, \,dx, \\
& = \Bigl[\frac{\sinh\sqrt{2\delta}} {\sqrt{2\delta}}\Bigr]^2
\int\limits_0^\infty  e^{-\delta x}\Bigl[x^q f_{S_2}(x)+qx^{q-1}{\bf P}(S_2<x)\Bigr] \,dx, \\
& = {\bf E}[S_2(\delta)^q] + \frac{q}{\delta}{\bf E}[T(\delta)^{q-1}].
\end{align}
\end{proof}
\begin{remark}
Why do Barnes beta distributions appear in the context of the Riemann zeta function? One possible answer is that they share one essential property in common with the zeta function -- they combine additive and multiplicative structures into a single object. This property follows from Definition \ref{bdef} and can be seen explicitly from Theorem \ref{BarnesFactorization}. 
\end{remark}
\section{Proofs of Theorems \ref{scaling}, \ref{multiplic}, and \ref{ShinFactor}.}
In this section we will give proofs of Theorems \ref{scaling}, \ref{multiplic}, and \ref{ShinFactor} using the Ruijsenaars representation of the log-gamma functions in Theorem \ref{R} and elementary properties of multiple Bernoulli polynomials.

\begin{proof}[Proof of Theorem \ref{scaling}]
Let $f(t)$ be defined by \eqref{fdef} and $\kappa>0.$ Then, it is easy to see from \eqref{fdef}
and \eqref{Bdefa} that we have the identities
\begin{gather}
f(t\,|\,\kappa\,a) = \kappa^{-M}\,f(\kappa\,t\,|\,a), \\
B_{M,m}(x\,|\,\kappa\,a) = \kappa^{m-M}\,B_{M,m}\bigl(\frac{x}{\kappa}\,|\,a\bigr).
\end{gather}
The result now follows by applying \eqref{key} to $L_M(\kappa\,w\,|\,\kappa\,a)$
changing variables $t'=\kappa\,t,$ and using Frullani's integral for logarithm.
\end{proof}

\begin{proof}[Proof of Theorem \ref{multiplic}]
Let $f(t)$ be defined by \eqref{fdef} and $k=1, 2, 3, \cdots.$ The key identity that we need is
\begin{equation}
\sum\limits_{p_1,\cdots,p_M=0}^{k-1}B_{M,m}\Bigl(w+\frac{\sum_{j=1}^M
p_j a_j}{k}\,|\,a\Bigr) = k^{M-m}\,B_{M,m}(kw\,|\,a).
\end{equation}
It is a simple corollary of \eqref{Bdefa}. By \eqref{key}, we then
have
\begin{gather}
\sum\limits_{p_1,\cdots,p_M=0}^{k-1}L_{M}\Bigl(w+\frac{\sum_{j=1}^M
p_j a_j}{k}\,|\,a\Bigr) = \int\limits_0^\infty \frac{dt}{t^{M+1}} \Bigl(
e^{-kwt}\,f(t) - \nonumber \\
- \sum\limits_{m=0}^{M-1} \frac{t^m}{m!}\,B_{M,m}(kw\,|\,a)
- \frac{t^M\,e^{-kt}}{M!}\, B_{M,M}(kw\,|\,a)\Bigr).
\end{gather}
By Frullani's integral for logarithm, we thus obtain
\begin{equation}
\sum\limits_{p_1,\cdots,p_M=0}^{k-1}L_{M}\Bigl(w+\frac{\sum_{j=1}^M
p_j a_j}{k}\,|\,a\Bigr) = L_M(kw\,|\,a) + \frac{1}{M!} B_{M,M}(kw\,|\,a)\log k.
\end{equation}
The result follows.
\end{proof}

The proof of Theorem \ref{ShinFactor} requires an auxiliary lemma of independent interest.
\begin{lemma}[Main Lemma]\label{mainlemma}
Let $f(t)$ be an arbitrary function of the Ruijsenaars class. Let $\Psi^{(f)}_{M+1}(x,y)$ be defined by \eqref{Psidef} with $B^{(f)}_{m}(x)$
given in \eqref{Bdefa}. Then,
\begin{equation}
\Psi^{(f)}_{M+1}(x,y) = \sum\limits_{m=0}^M \frac{B^{(f)}_{m}(x)}{m!}\frac{(-y)^{M-m}}{(M-m)!}\bigl(\sum\limits_{l=1}^{M-m}\frac{1}{l}-\log y\bigr) +\frac{1}{y(M+1)!}B^{(f)}_{M+1}(x). \label{Psiorig}
\end{equation}
Define the function
\begin{align}
\chi^{(f)}_{M+1}(x,y)\triangleq &\int\limits_0^\infty \frac{dt}{t^{M+1}}\Bigl[
\frac{f(t)}{e^{yt}-1}\,e^{-xt}-
\sum\limits_{m=0}^{M} \frac{t^m}{m!}B^{(f)}_{m}(x)\frac{1}{e^{yt}-1}-\nonumber \\
& - \frac{t^{M}}{y(M+1)!}B^{(f)}_{M+1}(x)e^{-yt}\Bigr], \label{chidef}
\end{align}
and let $\gamma$ denote Euler's constant. For any $\Re(x),\,\Re(y),\,\Re(w)>0 $ we have
\begin{gather}
\log\prod\limits_{k=1}^\infty \frac{\Gamma^{(f)}_M(w+ky)}{\Gamma^{(f)}_M(x+ky)}e^{\Psi^{(f)}_{M+1}(x,ky)-\Psi^{(f)}_{M+1}(w,ky)} = \chi^{(f)}_{M+1}(w,y)-\chi^{(f)}_{M+1}(x,y)+\nonumber \\ +\frac{\gamma}{y(M+1)!}\bigl(B^{(f)}_{M+1}(x)- B^{(f)}_{M+1}(w)\bigr).\label{shinproduct}
\end{gather}
\end{lemma}
\begin{proof}[Proof of Lemma \ref{mainlemma}]
The expression on the right-hand side of \eqref{Psiorig} is due to \cite{KataOhts}. Its equality to our definition of \(\Psi^{(f)}_{M+1}(x,y)\) in \eqref{Psidef} (which we only need here to
show that the integral in \eqref{Psidef} is convergent)
follows from the result of \cite{Ruij}
\begin{equation}\label{Ruijidentity}
\int\limits_0^\infty \frac{dt}{t^{N+1}} \Bigl[e^{-yt}-\sum\limits_{j=0}^{N-1} \frac{(-yt)^j}{j!} - \frac{(-yt)^N}{N!}e^{-t}\Bigr] = \frac{(-y)^{N}}{N!}\bigl(\sum\limits_{l=1}^{N}\frac{1}{l}-\log y\bigr).
\end{equation}
Indeed, applying \eqref{Ruijidentity} to the sum on the right-hand side of
\eqref{Psiorig}, we get
\begin{gather}
\int\limits_0^\infty \frac{dt}{t^{M+1}}\Bigl[\sum\limits_{m=0}^M \frac{t^m}{m!}B^{(f)}_{m}(x)e^{-yt} - \sum\limits_{m=0}^{M}\frac{t^m}{m!}B^{(f)}_{m}(x)\sum\limits_{j=0}^{M-m-1}
\frac{(-yt)^j}{j!} - \nonumber \\ - \sum\limits_{m=0}^{M}\frac{t^m}{m!}B^{(f)}_{m}(x)\frac{(-yt)^{M-m}}{(M-m)!}
e^{-t}\Bigr].
\end{gather}
Using the identity
\begin{equation}\label{Bbinom}
B^{(f)}_m(x+y) = \sum\limits_{p=0}^m \binom{m}{p}\,B^{(f)}_p(x) (-y)^{m-p},
\end{equation}
we can simplify
\begin{equation}
\sum\limits_{m=0}^{M}\frac{t^m}{m!}B^{(f)}_{m}(x)\frac{(-yt)^{M-m}}{(M-m)!} = \frac{t^M}{M!} B^{(f)}_{M}(x+y).
\end{equation}
Thus, to prove that \eqref{Psidef} and \eqref{Psiorig} are the same, it only remains to show
\begin{equation}
\sum\limits_{m=0}^{M}\frac{t^m}{m!}B^{(f)}_{m}(x)\sum\limits_{j=0}^{M-m-1}
\frac{(-yt)^j}{j!} = \sum\limits_{m=0}^{M-1}\frac{t^m}{m!}B^{(f)}_{m}(x+y),
\end{equation}
which follows from \eqref{Bbinom} by changes of the order of summation. 

Next, we need to prove that the integral in \eqref{chidef} is convergent. 
Define the function \(h(t)\triangleq \exp(-xt)\,f(t).\)
Then,
\begin{align}
\chi^{(f)}_{M+1}(x,y) = & \int\limits_0^\infty \frac{dt}{t^{M+1}}\Bigl[\Bigl(h(t)-\sum\limits_{m=0}^M \frac{t^m}{m!}h^{(m)}(0)\Bigr)\frac{1}{e^{yt}-1} - \nonumber \\ & 
- \frac{t^{M+1}}{(M+1)!}\frac{h^{(M+1)}(0)}{yt}e^{-yt}
\Bigr].
\end{align}
It is now easy to see that the integrand in bounded near $t=0$ and exponentially small as $t\rightarrow \infty$ so that the integral is convergent.  

We now proceed to establish \eqref{shinproduct}. The left-hand side of \eqref{shinproduct} can be evaluated using \eqref{key}.
Fix $K$ and substitute \eqref{key} and \eqref{Psidef}. After cancelations, we obtain
\begin{gather}
\log\prod\limits_{k=1}^K\frac{\Gamma^{(f)}_M(w+ky)}{\Gamma^{(f)}_M(x+ky)}e^{\Psi^{(f)}_{M+1}(x,ky)-\Psi^{(f)}_{M+1}(w,ky)} = \frac{B^{(f)}_{M+1}(x)- B^{(f)}_{M+1}(w)}{y(M+1)!}\sum\limits_{k=1}^K \frac{1}{k} + \nonumber \\ +
\int\limits_0^\infty \frac{dt}{t^{M+1}}\left[\Bigl(f(t)(e^{-wt}-e^{-xt})+\sum\limits_{m=0}^M \frac{t^m}{m!}\bigl(B^{(f)}_{m}(x)-B^{(f)}_{m}(w)\bigr)\Bigr)\frac{1-e^{-ytK}}{e^{yt}-1}\right].
\end{gather}
Using the identity
\begin{equation}
\sum\limits_{k=1}^K \frac{1}{k} = \gamma+\log K+O(1/K),\;K\rightarrow\infty,
\end{equation}
and Frullani's integral in the form
\begin{equation}
\log(K) = \int\limits_0^\infty
\bigl(e^{-yt}-e^{-ytK}\bigr)\frac{dt}{t},
\end{equation}
it remains to show that in the limit of $K\rightarrow\infty$
\begin{gather}
\lim\limits_{K\rightarrow\infty}\int\limits_0^\infty \frac{dt}{t^{M+1}}\Bigl[\Bigl(f(t)(e^{-wt}-e^{-xt})+\sum\limits_{m=0}^M \frac{t^m}{m!}\bigl(B^{(f)}_{m}(x)-B^{(f)}_{m}(w)\bigr)\Bigr)\frac{1}{e^{yt}-1}
+\nonumber \\ + \frac{t^{M+1}}{(M+1)!}\frac{\bigl(B^{(f)}_{M+1}(x)- B^{(f)}_{M+1}(w)\bigr)}{yt}
\Bigr]e^{-ytK}=0. \label{Klim}
\end{gather}
Define the function \(h(t)\triangleq \exp(-wt)\,f(t).\)
To prove \eqref{Klim}, given the definition of $B^{(f)}_{m}(w)$
in \eqref{Bdefa}, it is sufficient to show that
\begin{equation}
\lim\limits_{K\rightarrow\infty}\int\limits_0^\infty \frac{dt}{t^{M+1}}\Bigl[\Bigl(h(t)-\sum\limits_{m=0}^M \frac{t^m}{m!}h^{(m)}(0)\Bigr)\frac{1}{e^{yt}-1}
- \frac{t^{M+1}}{(M+1)!}\frac{h^{(M+1)}(0)}{yt}
\Bigr]e^{-ytK}=0.
\end{equation}
It is elementary to check that the integrand is bounded for all $t\geq 0$ so that the result follows by dominated convergence.
\end{proof}

\begin{proof}[Proof of Theorem \ref{ShinFactor}]
Let $f(t)$ be defined by \eqref{fgdef} for some $g(t)$ of Ruijsenaars class.
Consider the quantity
\begin{equation}\label{Pdef}
P^{(f)}_{M+1}(w, y\,|\,a) \triangleq \log\Gamma^{(f)}_{M+1}\bigl(w\,|\,a, y\bigr) - \chi^{(f)}_{M+1}(w,y\,|\,a) - \log\Gamma^{(f)}_M(w\,|\,a),
\end{equation}
where \(a\triangleq (a_1\cdots a_M)\) as usual. By \eqref{key}, \eqref{fgdef}, \eqref{chidef}, we have the identity
\begin{align}
P^{(f)}_{M+1}(w, y\,|\,a) & = \int\limits_0^\infty \frac{dt}{t^{M+2}}\Bigl[
\sum\limits_{m=0}^{M-1} \frac{t^m}{m!}B^{(f)}_{M,m}(w\,|\,a)\frac{t}{1-e^{-yt}}
+ \nonumber \\ & + \frac{t^M}{M!}B^{(f)}_{M,M}(w\,|\,a)\frac{t}{e^{yt}-1} +
\frac{t^{M+1}}{M!}e^{-t}B^{(f)}_{M,M}(w\,|\,a) + \nonumber \\ & +
\frac{t^{M+1}}{y(M+1)!}e^{-yt}B^{(f)}_{M,M+1}(w\,|\,a) -
\sum\limits_{m=0}^{M} \frac{t^m}{m!}B^{(f)}_{M+1,m}(w\,|\,a,y) - \nonumber \\ & -
\frac{t^{M+1}}{(M+1)!}e^{-t}B^{(f)}_{M+1,M+1}(w\,|\,a,y)\Bigr]. \label{Pformula}
\end{align}
We now have two expressions for $\chi^{(f)}_{M+1}(w,y\,|\,a)$ given by \eqref{shinproduct} and \eqref{Pdef}. Hence, letting $y=a_{M+1},$ we obtain the identity
\begin{align}
\Gamma^{(f)}_{M+1}\bigl(w\,|\,a,a_{M+1}\bigr) = & \prod\limits_{k=1}^\infty \frac{\Gamma^{(f)}_M(w+ka_{M+1}\,|\,a)}{\Gamma^{(f)}_M(x+ka_{M+1}\,|\,a)}e^{\Psi^{(f)}_{M+1}(x,ka_{M+1}\,|\,a)-\Psi^{(f)}_{M+1}(w,ka_{M+1}\,|\,a)} \star \nonumber \\
& \star \exp\Bigl(\phi^{(f)}_{M+1}(w,x\,|\,a, a_{M+1})\Bigr)\Gamma^{(f)}_{M}(w\,|\,a),
\end{align}
where the quantity \(\phi^{(f)}_{M+1}(w,x\,|\,a, a_{M+1})\) is defined by
\begin{align}
\phi^{(f)}_{M+1}(w,x\,|\,a, a_{M+1}) \triangleq & P^{(f)}_{M+1}(w, a_{M+1}\,|\,a) +
\chi^{(f)}_{M+1}(x,a_{M+1}\,|\,a) + \nonumber \\ & +\frac{\gamma}{a_{M+1}(M+1)!}\bigl(B^{(f)}_{M,M+1}(w\,|\,a) - B^{(f)}_{M,M+1}(x\,|\,a)\bigr). \label{phiformula}
\end{align}
It remains to note that $\Psi^{(f)}_{M+1}(w,y\,|\,a)$ is a polynomial in $w$ of degree $M+1$ by construction, $P^{(f)}_{M+1}(w, y\,|\,a)$ is a polynomial in $w$ of degree $M+1$ by \eqref{Pformula}, and so is $\phi^{(f)}_{M+1}(w,x\,|\,a, a_{M+1})$ by \eqref{phiformula}.
\end{proof}

\section{Conclusions}
We advanced the general theory of Barnes beta probability distributions by showing that the scaling invariance, multiplication formula, and Shintani factorization of Barnes multiple gamma functions imply novel properties of the Barnes beta distributions. In particular, we derived a novel infinite product factorization of the Mellin transform (Barnes factorization) as a corollary of the Shintani factorization and used it to derive explicit formulas for integral moments and mass at 1. 

We considered the probability distribution underlying the Selberg integral (conjectured to be the law of the total mass of the limit lognormal stochastic measure on the unit interval) as the main area of applications. We used the multiplication formula to show that a certain combination of products of ratios of Barnes double gamma functions is the Mellin transform of a probability distribution that splits into a product of a lognormal and three Barnes beta distributions $\beta^{-1}_{2,2}.$ We used this result to give a new, purely probabilistic proof of the existence and structure of the Selberg integral distribution and then established involution invariance of its components by applying the general scaling invariance property of Barnes beta distributions. In particular, we gave a new probabilistic proof of the involution invariance of the corresponding Mellin transform that was first noted in \cite{FLDR}. It is interesting to point out that the Selberg integral distribution can be naturally thought of as consisting of two blocks of factors that are intrinsically different: a lognormal and three $\beta^{-1}_{2,2}$ on the one hand and a Fr\'echet on the other. The first block is involution invariant, whereas the other is not. Given the conjecture of \cite{FyoBou} that the Fr\'echet factor is the law of the total mass of the limit lognormal measure on the circle, we can naturally interpret the first block as being the superstructure that is needed to ``lift'' the law of this measure from the circle to the unit interval.

We considered the weak limit of the Selberg integral distribution that corresponds to the critical limit lognormal measure and described in detail the resulting probability distribution (conjectured to be the law of the derivative martingale). We noted that its Mellin transform can
be represented in the form of a product of ratios of Barnes $G$ functions
and factorized it into a product of a lognormal, \(\beta^{-1}_{2, 2}\), Pareto, and Fr\'echet distributions. The L\'evy density of \(-\log\beta_{2, 2}\) was shown to be related to the Laplace transform of the $C_2$ distribution, which allowed us to compute the cumulants of a related class
of \(-\log\beta_{2, 2}\) distributions in terms of the values of the Riemann xi functions at the integers.

We contributed to the probabilistic theory of the Riemann xi function. We showed that the Jacobi triple product has a probabilistic interpretation in terms of a limit of Barnes beta distributions and constructed a one-parameter family of
probability distributions in this limit. This construction resulted in a functional equation for entire functions expressing certain sums over the values of xi in terms of the Mellin transform of the limit distribution. In  particular, we expressed the xi function itself by taking the zero limit of the parameter of the limit distribution.

We have reviewed certain aspects of the general theory of Barnes multiple gamma functions that are relevant to Barnes beta distributions using the approach
(and normalization) of Ruijsenaars. In particular, we stated and proved a more general version of the Shintani identity than what has been done previously
and applied it to explain the origin of the infinite product factorization of
the Mellin transform of Barnes beta distributions.

Overall, our results suggest that there might be an interesting, not yet understood, connection between the Selberg integral distribution, especially in the critical case, and the Riemann xi. We have shown that the two objects have
Barnes beta distributions as their respective building blocks (\(\beta_{2, 2}\)
in the former case and a limit of \(\beta_{2M, 3M}\) in the latter) and so are naturally led to speculate that they might be related directly, perhaps involving the $G$ function. The relationship between the cumulants of  \(-\log\beta_{2, 2}\) in the critical case and values
of the xi function at the integers is a first step in this direction, which suggests that there might be a direct relationship between  \(-\log\beta_{2, 2}\) and $C_2$ or $S_2$ distributions and that xi might be somehow related
to the Mellin transform of \(-\log\beta_{2, 2}\).

\ACKNO{The author wishes to thank the organizers and participants of
the conference ``Branching Diffusions and Gaussian Free Fields in Physics, Probability and Number Theory'', Marseille, 2013 for 
drawing out interest to the critical limit lognormal measure. The author also wants to extend special gratitude to Y. V. Fyodorov, who made it possible
for the author to attend the event and brought reference \cite{FyoBou} to our attention. Finally, the author wants to thank the anonymous referee for the careful reading of the paper and helpful suggestions.
}


\begin{thebibliography}{50}

\bibitem{Genesis} E. W. Barnes, The genesis of the double gamma functions, \emph{Proc. London Math. Soc.} \textbf{ s1-31}, (1899),
358--381.

\bibitem{mBarnes} E. W. Barnes, On the theory of the multiple gamma
function. \emph{Trans. Camb. Philos. Soc.} \textbf{19}, (1904),
374--425.

\bibitem{barraletal} J. Barral, A. Kupiainen, M. Nikula, E. Saksman, C. Webb, Basic properties of critical lognormal multiplicative chaos, preprint (2013),
\url{http://arxiv.org/abs/1303.4548}.

\bibitem{BiaPitYor} P. Biane, J. Pitman, M. Yor, Probability laws
related to the Jacobi theta and Riemann zeta functions, and brownian
excursions. \emph{Bulletin of the American Mathematical Society}
\textbf{38}, (2001), 435--465.

\bibitem{Char} T. M. Bisgaard, Z. Sasvari,
\emph{Characteristic Functions and Moment Sequences}, Nova Science
Publishers, Huntington, 2000.

\bibitem{Bressoud} D. M. Bressoud, Combinatorial analysis, in: \emph{NIST Handbook of Mathematical Functions}, NIST and Cambridge University Press, (2010), 617--636. 

\bibitem{ChaLet}  J-F. Chamayou, G. Letac, Additive properties of the Dufresne laws and their multivariate extension. \emph{J. Theoret. Probab.} \textbf{12},
(1999), 1045--1066.

\bibitem{Duf10} D. Dufresne, $G$ distributions and the beta-gamma
algebra. \emph{Elect. J. Probab.} \textbf{15}, (2010), 2163--2199.

\bibitem{dupluntieratal} B. Duplantier, R. Rhodes, S. Sheffield, V. Vargas,
Critical gaussian multiplicative chaos: convergence of the derivative martingale, preprint (2012), \url{http://arxiv.org/abs/1206.1671}.

\bibitem{ForresterBook} P. J. Forrester,  \emph{Log-Gases and Random
Matrices}, Princeton University Press, Princeton, 2010.

\bibitem{FyoBou} Y. V. Fyodorov and J. P. Bouchaud: Freezing and extreme-value statistics in a random energy model with logarithmically correlated potential.
\emph{J. Phys. A, Math Theor.} \textbf{41}, (2008), 372001.

\bibitem{YK} Y. V. Fyodorov and J. P. Keating: Freezing transitions and extreme values: random matrix theory, $\zeta(1/2+it),$ and disordered landscapes. \emph{Phil. Trans. R. Soc. A} \textbf{372}, (2014), 20120503. 

\bibitem{FLDR} Y. V. Fyodorov, P. Le Doussal, A. Rosso, Statistical mechanics of logarithmic REM: duality, freezing and extreme value statistics of 1/f noises
generated by gaussian free fields. \emph{J. Stat. Mech.}, (2009), P10005.

\bibitem{YM} Y. V. Fyodorov and A. D. Mirlin: Level-to-level 
fluctuations in the inverse participation ratio in finite quasi 1D disordered systems. \emph{Phys. Rev. Lett.} \textbf{71}, (1993), 412--415.

\bibitem{HacKuz} D. Hackmann, A. Kuznetsov, A note on the series representation for the density of the supremum of a stable process, \emph{Elect. Comm. in Probab.} \textbf{18}, no. 42, (2013), 1--5.

\bibitem{HubKuz} F. Hubalek, A. Kuznetsov, A convergent series representation for the density of the supremum of a stable process, \emph{Elect. Comm. in Probab.} \textbf{16}, (2011), 84--95.

\bibitem{Jacod} J. Jacod, E. Kowalski, A. Nikeghbali,
Mod-gaussian convergence: new limit theorems in probability and
number theory. \emph{Forum Mathematicum} \textbf{23}, (2011),
835--873.

\bibitem{KataOhts} K. Katayama, M. Ohtsuki, On the multiple
gamma-functions. \emph{Tokyo J. Math.} \textbf{21}, (1998),
159--182.

\bibitem{KS} J.P. Keating, N.C. Snaith (2000), Random matrix theory and
$\zeta(1/2+it),$ \emph{Comm. Math. Phys.} \textbf{214}: 57--89.

\bibitem{Kuz} A. Kuznetsov, On extrema of stable processes. \emph{Ann. Probab.} \textbf{39}, (2011), 1027--1060.

\bibitem{LagRains} J. Lagarias, E. Rains, On a two-variable zeta function for number fields. \emph{Ann. Inst. Fourier, Grenoble} \textbf{53}, (2003),
1--68.

\bibitem{NikYor} A. Nikeghbali, M. Yor, The Barnes G function
and its relations with sums and products of generalized gamma
convolutions variables. \emph{Elect. Comm. in Prob.} \textbf{14},
(2009), 396--411.

\bibitem{Me4} D. Ostrovsky, Mellin transform of the limit
lognormal distribution. \emph{Comm. Math. Phys.} \textbf{288},
(2009), 287--310.

\bibitem{Me} D. Ostrovsky, Selberg integral as a meromorphic
function. \emph{Int. Math. Res. Notices}, \textbf{17}, (2013), 3988--4028.

\bibitem{Me13} D. Ostrovsky, Theory of Barnes beta distributions. \emph{Elect. Comm. in Prob.}, \textbf{18}, no. 59, (2013), 1--16.

\bibitem{PitYor} J. Pitman, M. Yor, Infinitely divisible laws
associated with hyperbolic functions. \emph{Canad. J. Math.}
\textbf{55}, (2003), 292--330.

\bibitem{Ruij} S. N. M. Ruijsenaars, On Barnes' multiple zeta and
gamma functions. \emph{Advances in Mathematics} \textbf{156},
(2000), 107--132.

\bibitem{Selberg} A. Selberg, Remarks on a multiple integral. \emph{Norske Mat.
Tidsskr.} \textbf{26}, (1944), 71--78.

\bibitem{Shintani} T. Shintani, A proof of the classical
Kronecker limit formula. \emph{Tokyo J. Math.} \textbf{3}, (1980),
191--199.

\bibitem{SriCho} H. M. Srivastava, J. Choi, \emph{Zeta and q-Zeta Functions and Associated Series and Integrals}, Elsevier, Amsterdam, 2012.

\bibitem{WhitWat}  E. T. Whittaker, G. N. Watson, \emph{A
Course of Modern Analysis}, 4th ed., Cambridge University Press,
London, 1958.
\end{thebibliography}
\end{document}